\documentclass[11pt]{amsart}

\usepackage{amsmath, amssymb}
\usepackage{amsfonts}
\usepackage{mathrsfs}
\usepackage[arrow,matrix,curve,cmtip,ps]{xy}
\usepackage{color}
\usepackage{todonotes}

\usepackage{float}

\usepackage{etex}
\usepackage{tikz}
\usetikzlibrary{arrows,positioning,scopes,calc,intersections}

\usepackage{amsthm}
\usepackage{mathtools} %math arrows
\usepackage{array}
\usepackage{mathdots}

\usepackage[top= 2.25cm, bottom=2cm, left = 2cm, right= 2cm]{geometry}
\usepackage{hyperref}
\usepackage{setspace}

\allowdisplaybreaks

\numberwithin{equation}{section}

\newtheorem{theorem}{Theorem}[section]
\newtheorem{lemma}[theorem]{Lemma}
\newtheorem{proposition}[theorem]{Proposition}
\newtheorem{corollary}[theorem]{Corollary}

\newtheorem*{theorem*}{Theorem}
\theoremstyle{remark}
\newtheorem{remark}[theorem]{Remark}

\numberwithin{equation}{section}

\newcommand{\q}{\psi}

\newcommand{\Pkt}[1]{\Pi_{#1}}

\newcommand{\cPkt}[1]{\bar{\Pi}_{#1}}

\renewcommand{\r}{\pi}

\newcommand{\D}[1]{\widehat{#1}}

\newcommand{\Jac}{\text{Jac}}

\newcommand{\ul}{\underline{l}}

\newcommand{\ueta}{\underline{\eta}}

\begin{document}
\title[]
{Nonarchimedean components of non-endoscopic automorphic representations for quasisplit $Sp(N)$ and $O(N)$}

\author{Bin Xu}

\address{Yau Mathematical Sciences Center and Department of Mathematics \\  Tsinghua University, Beijing, China}
\email{binxu@tsinghua.edu.cn}

% 'MSC classification, keywords and grant acknowledgements'

%\date{\today}

\maketitle

\begin{abstract}
Arthur classified the discrete automorphic representations of symplectic and orthogonal groups over a number field by that of the general linear groups. In this classification, those that are not from endoscopic lifting correspond to pairs $(\phi, b)$, where $\phi$ is an irreducible unitary cuspidal automorphic representation of some general linear group and $b$ is an integer. In this paper, we study the local components of these automorphic representations at a nonarchimedean place, and we give a complete description of them in terms of their Langlands parameters.
\end{abstract}

\section{Introduction}

Let $G$ be a split symplectic or special odd orthogonal group over a number field $k$. Arthur \cite{Arthur:2013} proved the automorphic representations of $G(\mathbb{A}_{k})$ can be parametrized by the global Arthur parameters, which are isobaric sums
\[
\psi = \boxplus_{i} (\phi_{i} \boxtimes \nu_{b_{i}}),
\]
where $\phi_{i}$ is certain irreducible unitary cuspidal automorphic representation of a general linear group and $\nu_{b_{i}}$ is the $(b_{i} - 1)$-th symmetric power representation of $SL(2, \mathbb{C})$. For any such $\psi$, Arthur attached a global Arthur packet $\Pkt{\q}$, which is a multi-set of isomorphism classes of irreducible admissible representations of $G(\mathbb{A}_{k})$. This packet admits a restricted tensor product decomposition
\[
\Pkt{\q} := \otimes_{v} \Pkt{\q_{v}}
\]
where we denote by $\q_{v}$ the local component of $\q$ at each place $v$, and $\Pkt{\q_{v}}$ is a multi-set of isomorphism classes of irreducible admissible representations of $G(k_{v})$, called local Arthur packet. By the local Langlands correspondence for general linear groups \cite{HarrisTaylor:2001} \cite{Henniart:2000} \cite{Scholze:2013} \cite{Langlands:1989}, we can associate $\phi_{i, v}$ with a representation of the Weil-Deligne group $WD_{k_{v}} := W_{k_{v}} \times SL(2, \mathbb{C})$ at the nonarchimedean places (resp. $W_{k_{v}}$ at the archimedean places), which will still be denoted by $\phi_{i, v}$. Then $\q_{v}$ can be viewed as a representation of $WD_{k_{v}} \times SL(2, \mathbb{C})$ at the nonarchimedean places (resp. $W_{k_{v}} \times SL(2, \mathbb{C})$ at the archimedean places). In particular, Arthur showed that it factors through the Langlands dual group of $G(k_{v})$. We will call $\q_{v}$ a local Arthur parameter for $G(k_{v})$. In this paper, we would like to describe the Langlands parameters of the elements inside $\Pkt{\q_{v}}$, when $\q$ consists of a single term, i.e., 
\begin{align}% single Eq
\label{eq: single}
\q = \phi \boxtimes \nu_{b}
\end{align}
and $v$ is a nonarchimedean place. It follows from Arthur's theory \cite{Arthur:2013} that the representations in such $\Pkt{\q}$ do not come from endoscopic lifting, so this justifies our title.

From now on, we will let $G$ be a split symplectic or special odd orthogonal group over a $p$-adic field $F$. Let $\D{G}$ be the complex dual group of $G$. We recall an Arthur parameter for $G(F)$ is a $\D{G}$-conjugacy class of admissible homomorphisms
\[
\psi: W_{F} \times SL(2, \mathbb{C}) \times SL(2, \mathbb{C}) \rightarrow \D{G}
\]
with the property that $\psi(W_{F})$ is bounded. By composing with the standard representation of $\D{G}$, we can view $\psi$ as a representation of $W_{F} \times SL(2, \mathbb{C}) \times SL(2, \mathbb{C})$. It decomposes as 
\begin{align}% local A-parameter Eq
\label{eq: local A-parameter}
\psi = \oplus_{i = 1}^{n} \, \rho_{i} \otimes \nu_{a_{i}} \otimes \nu_{b_{i}}
\end{align}
where $\rho_{i}$ is an irreducible unitary representation of $W_{F}$ and $a_{i}, b_{i} \in \mathbb{Z}$. To describe the associated packet $\Pkt{\q}$, we will take $\rho_{i}$ to be the corresponding irreducible supercuspidal representation of $GL(d_{\rho_{i}}, F)$ through the local Langlands correspondence. Then we can construct a self-dual representation of $GL(N, F)$ by
\[
\pi^{GL}_{\q} := \times_{i = 1}^{n} \, Sp(St(\rho_{i}, a_{i}), b_{i}),
\]
which is an induction of the Speh representations. Recall the Steinberg representation $St(\rho_{i}, a_{i})$ is the unique irreducible subrepresentation of the induction
\[
\rho_{i}||^{(a_{i} - 1)/2} \times \rho_{i}||^{(a_{i} - 3)/2} \cdots \times \rho_{i}||^{-(a_{i} - 1)/2}
\]
and the Speh representation $Sp(St(\rho_{i}, a_{i}), b_{i})$ is the unique irreducible subrepresentation of 
\begin{align}% Speh Eq
\label{eq: Speh}
St(\rho_{i}, a_{i})||^{-(b_{i} - 1)/2} \times St(\rho_{i}, a_{i})||^{-(b_{i} - 3)/2} \cdots \times St(\rho_{i}, a_{i})||^{(b_{i} - 1)/2}
\end{align}
We will also denote the Steinberg representation by 
\[
\langle (a_{i} - 1)/2, \cdots, -(a_{i} - 1)/2\rangle
\] 
and the Speh representation by a matrix 
\begin{align}% matrix Eq
\label{eq: matrix}
\begin{pmatrix}
              (a_{i} - b_{i})/2 & \cdots & 1 - (a_{i} + b_{i})/2 \\
              \vdots &  & \vdots \\
              (a_{i} + b_{i})/2 - 1 & \cdots & -(a_{i} - b_{i})/2
\end{pmatrix}
\end{align}
where each row corresponds to the exponents of the shifted Steinberg representations in \eqref{eq: Speh}. Since $\pi^{GL}_{\q}$ is self-dual, one can consider its twisted character. Arthur \cite{Arthur:2013} proved that there exists a stable finite linear combination of characters on $G(F)$, whose twisted endoscopic transfer is this twisted character. By the linear independence of characters, this determines $\Pkt{\q}$ as a finite subset of isomorphism classes of irreducible admissible representations of $G(F)$. (Moeglin \cite{Moeglin1:2011} proved the Arthur packet is always multiplicity free in this case.) However, this does not tell us explicitly which representations are contained in it. To answer this question, we need a parametrization of the set ${\rm Irr}(G(F))$ of isomorphism classes of irreducible admissible representations of $G(F)$. This is given by the local Langlands correspondence for $G(F)$. Arthur \cite{Arthur:2013} proved that there is a canonical bijection (after fixing a Whittaker datum)  
\[
{\rm Irr}(G(F)) \cong \{(\phi, \epsilon) | \phi \in \Phi(G(F)), \epsilon \in {\rm Irr}(\mathcal{S}_{\phi})\},
\]
where $\Phi(G(F))$ is the set of Langlands parameters, which are $\D{G}$-conjugacy classes of admissible homomorphisms 
\[
\phi: W_{F} \times SL(2, \mathbb{C}) \rightarrow \D{G},
\]
and
\[
\mathcal{S}_{\phi} := \pi_{0}(Z_{\D{G}}(\phi)/Z(\D{G}) ),
\]
where $Z_{\D{G}}(\phi)$ is the stabilizer of $\phi$ in $\D{G}$ and $Z(\D{G})$ is the center of $\D{G}$. We will call the pair $(\phi, \epsilon)$ complete Langlands parameter of $G(F)$, and denote the corresponding representation by $\pi(\phi, \epsilon)$.

Back to the Arthur parameter \eqref{eq: local A-parameter}, let us write $A_{i} = (a_{i} + b_{i}) / 2 -1$, $B_{i} = |a_{i} - b_{i}|/2$ and $\zeta_{i} = {\rm sgn}(a_{i} - b_{i})$. When $a_{i} = b_{i}$, we may choose $\zeta_{i}$ arbitrarily. We will call $(\rho, a_{i}, b_{i})$ or $(\rho, A_{i}, B_{i}, \zeta_{i})$ Jordan blocks, and denote the set of Jordan blocks by $Jord(\q)$. For simplicity, we will assume that 
\begin{align}% fix rho Eq
\label{eq: fix rho}
\text{ $\rho_{i} = \rho$ for some fixed $\rho$, and $(\rho, a_{i}, b_{i})$ all have the same parity as $\D{G}$. }
\end{align}
Since we want to study the local component of a global Arthur parameter of the type \eqref{eq: single}, we can assume all $b_{i}$ are equal and denote it by $b$. So we may rewrite \eqref{eq: local A-parameter} as
\begin{align}% fix b Eq
\label{eq: fix b}
\psi = \oplus_{i = 1}^{n} \, \rho \otimes \nu_{a_{i}} \otimes \nu_{b}
\end{align}
Under our assumptions, all $a_{i}$ will have the same parity. The simplest case is when $\psi$ consists of a single term, i.e.,
\[
\psi = \rho \otimes \nu_{a} \otimes \nu_{b}
\]
and $a \geqslant b$. In this case, we have the following result due to M{\oe}glin \cite[Theorem 4.2]{Moeglin:2009}. Firstly, there is a bijection 
\[
\Pkt{\q} \rightarrow \{(l, \eta) \in \mathbb{Z} \times \{\pm 1\} \, | \, 0 \leqslant l \leqslant [(A - B + 1)/2] \text{ and } \epsilon_{l, \eta} = 1\} / \sim
\]
where
\begin{align}% sign Eq
\label{eq: sign}
\epsilon_{l, \eta} := \eta^{A - B + 1}(-1)^{[(A - B + 1)/2] + l}
\end{align}
and the equivalence relation $\sim$ only identifies those $(l, \eta)$ and $(l', \eta')$ for $l = l' = (A - B + 1)/2$. Secondly, the representation $\pi(\psi, l, \eta)$ parametrized by $(l, \eta)$ satisfies
\begin{align}
\label{eq: Moeglin}% Moeglin Eq
\pi(\psi, l, \eta) \hookrightarrow \begin{pmatrix}
              B & \cdots & -A \\
              \vdots &  & \vdots \\
              B + l - 1 & \cdots & -(A - l + 1) \end{pmatrix} \rtimes \pi(\phi', \epsilon')
\end{align}
as the unique irreducible subrepresentation. Here the matrix represents a shifted Speh representation (cf. \eqref{eq: matrix}) and $\pi(\phi', \epsilon')$ is a discrete series representation of $G'(F)$, which is of the same type as $G(F)$. The complete Langlands parameter $\phi_{l}$ of the shifted Speh representation factors through that of the inducing representation (cf. \eqref{eq: Speh}), i.e., 
\[
\phi_{l} = \oplus_{i = 0}^{l-1} \, (\rho||^{i - (A-B)/2} \otimes \nu_{A + B+1})
\]
and
\[
\phi' = \oplus_{C = B + l}^{A-l} \, \rho \otimes \nu_{2C+1}.
\]
We can view $\phi'$ as an Arthur parameter, where the second $SL(2, \mathbb{C})$ maps trivially. Then its Jordan blocks are $(\rho, C, C, +)$ for $B+l \leqslant C \leqslant A-l$. The character $\epsilon'$ can be represented by a sign function over this set of Jordan blocks. In this way, we have
\[
\epsilon'(\rho, C, C, +) = (-1)^{C - (B + l)}\eta.
\]
The sign condition $\epsilon_{l, \eta} = 1$ guarantees that
\[
\prod_{C = B + l}^{A - l} \epsilon'(\rho, C, C, +) = 1,
\]
which is the necessary condition for $\epsilon'$ to define a character of $\mathcal{S}_{\phi'}$.
One can also describe the complete Langlands parameter $(\phi, \epsilon)$ for $\pi(\psi, l, \eta)$ from the embedding \eqref{eq: Moeglin}. Indeed, 
\[
\phi = \phi_{l} \oplus \phi' \oplus \phi_{l}^{\vee}.
\]
Moreover, $\epsilon$ corresponds to $\epsilon'$ under the natural isomorphisms
\(
\mathcal{S}_{\phi} \cong \mathcal{S}_{\phi'}.
\)
By M{\oe}glin's result, one can also view $\pi(\phi', \epsilon')$ as an element in $\Pkt{\q'}$, where $\q'$ is the Arthur parameter of $G'(F)$ consisting of only one Jordan block $(\rho, A-l, B+l, +)$. In particular, 
\begin{align}% discrete series in A-packet Eq
\label{eq: discrete series in A-packet}
\pi(\phi', \epsilon') = \pi(\psi', l', \eta') 
\end{align}
for $l' = 0$ and $\eta' = \epsilon'(\rho, B+l, B+l, +)$. To save notations, we will write
\[
\pi(\psi', l', \eta') = \pi((\rho, A-l, B+l, 0, \eta', +)).
\]
%Moeglin's result will be the model for what we are going to describe below.
In general, we can divide the Jordan blocks in \eqref{eq: fix b} into two classes.
\begin{itemize}

\item $a_{i} \geqslant b$, i.e., $\zeta_{i} = +$: $A_{i} - B_{i} = b -1, A_{i} + B_{i} = a_{i} - 1$. So all intervals $[B_{i}, A_{i}]$ have the same length and are centered beyond $(b-1)/2$.

\item $a_{i} < b$, i.e., $\zeta_{i} = -$: $A_{i} - B_{i} = a_{i} - 1, A_{i} + B_{i} = b - 1$. So all intervals are centered at $(b-1)/2$.

\end{itemize}
We reorder the Jordan blocks such that $A_{i} \geqslant A_{i-1}$. Then there exists an integer $m$ such that $(\rho, A_{i}, B_{i}, \zeta_{i})$ is in the first class if $i > m$ and the second class if $i \leqslant m$. Now we can state our main results.

\subsection{Reductions}

\begin{theorem}% push integral THEOREM
\label{thm: push integral}

Suppose $A_{i}, B_{i} \in \mathbb{Z}$. Let $\q'$ be obtained from $\q$ by replacing all $(\rho, A_{i}, B_{i}, \zeta_{i})$ by $(\rho, A'_{i}, B'_{i}, \zeta'_{i})$ such that: $\zeta'_{i} = +$ and 
\begin{itemize}

\item $A'_{i} = A_{i}, B'_{i} = B_{i}$ for $i > m$;

\item $A'_{i} = A_{i} - B_{i}, B'_{i} = 0$ for $i \leqslant m$.

\end{itemize}
Then there is a bijection 
\[
\Pkt{\q} \rightarrow \Pkt{\q'}, \quad \pi \mapsto \pi'
\]
such that any representation $\pi \in \Pkt{\q}$ is given as the unique irreducible subrepresentation of 
\[
\pi \hookrightarrow \times_{i \leqslant m} \begin{pmatrix}
              -B_{i} & \cdots & -A_{i} \\
              \vdots &  & \vdots \\
              - 1 & \cdots & -(A_{i} - B_{i} +1)
       \end{pmatrix} \rtimes \pi'
\]
for the corresponding $\pi' \in \Pkt{\q'}$. Moreover, if $(\phi', \epsilon')$ is the complete Langlands parameter of $\pi'$, then the complete Langlands parameter $(\phi, \epsilon)$ of $\pi$ is given as follows: 
\[
\phi = (\oplus_{i \leqslant m}  \phi_{i}) \oplus \phi' \oplus (\oplus_{i \leqslant m}  \phi_{i}^{\vee}),
\]
where $\phi_{i}$ is the Langlands parameter of the corresponding shifted Speh representation and $\epsilon$ corresponds to $\epsilon'$ under the natural isomorphism
\(
\mathcal{S}_{\phi} \cong \mathcal{S}_{\phi'}.
\)

\end{theorem}

\begin{theorem}% push half-integral THEOREM
\label{thm: push half-integral}
Suppose $A_{i}, B_{i} \notin \mathbb{Z}$. We consider the maximal sequence of integers 
\[
0 = s_{0} < s_{1} < \cdots < s_{l} = m
\]
such that $A_{s_{j}} - B_{s_{j}} \neq A_{s_{j} + 1} - B_{s_{j} + 1}$. For any $0 \leqslant k \leqslant l$, we get a new parameter $\q'_{k}$ by replacing all $(\rho, A_{i}, B_{i}, \zeta_{i})$ by $(\rho, A'_{i}, B'_{i}, \zeta'_{i})$ such that: $\zeta'_{i} = +$ and
\begin{itemize}

\item $A'_{i} = A_{i}, B'_{i} = B_{i}$ for $i > m$;

\item $A'_{i} = A_{i} - B_{i} -1/2, B'_{i} = 1/2$ for $i \leqslant m$ and $i \neq s_{k}$;

\item $A'_{i} = A_{i} - B_{i} + 1/2, B'_{i} = 1/2$ for $i = s_{k}$.

\end{itemize}
Then we can divide $\Pkt{\q}$ into $l + 1$ classes, i.e.,
\[
\Pkt{\q} = \sqcup_{k = 0}^{l} \, \Pkt{\q}(k),
\]
and for any $0 \leqslant k \leqslant l$, we can get an injection 
\[
\Pkt{\q}(k) \hookrightarrow \Pkt{\q'_{k}}, \quad \pi \mapsto \pi(\q'_{k}, \ul', \ueta'),
\]
such that
\begin{align*}
\pi \hookrightarrow \times_{s_{k} \neq i \leqslant m} \begin{pmatrix}
              -B_{i} & \cdots & -A_{i}\\
              \vdots &  & \vdots \\
              -1/2  & \cdots & -(A_{i} - B_{i} +1/2) \end{pmatrix} \times \begin{pmatrix}
              -B_{s_{k}} & \cdots & -A_{s_{k}} \\
              \vdots &  & \vdots \\
              -3/2 & \cdots & -(A_{s_{k}} - B_{s_{k}} +3/2)
       \end{pmatrix} \rtimes \pi(\q'_{k}, \ul', \ueta')
\end{align*}
as the unique irreducible subrepresentation. Here we have parametrized the elements of $\Pkt{\q'_{k}}$ by $(\ul', \ueta')$ as explained in Section \ref{subsec: special case} below. The image is characterized by the condition that for all $i \leqslant s_{k}$,
\begin{itemize}

\item $l'_{i} = 0$;

\item $\eta'_{i} = - \prod_{j < i} (-1)^{A_{j} - B_{j} + 1}$.

\end{itemize}
When $k \neq 0$, the second condition can also be simplified as $\eta'_{1} = -1$. Moreover, if $(\phi', \epsilon')$ is the complete Langlands parameter of $\pi(\q'_{k}, \ul', \ueta')$, then the complete Langlands parameter $(\phi, \epsilon)$ of $\pi$ is given as follows: 
\[
\phi = (\oplus_{i \leqslant m}  \phi_{i}) \oplus \phi' \oplus (\oplus_{i \leqslant m}  \phi_{i}^{\vee}),
\]
where $\phi_{i}$ is the Langlands parameter of the corresponding shifted Speh representation and $\epsilon$ corresponds to $\epsilon'$ under the natural isomorphism
\(
\mathcal{S}_{\phi} \cong \mathcal{S}_{\phi'}.
\)

\end{theorem}

\subsection{A special case}% special case SUBSECTION
\label{subsec: special case}

The previous two theorems reduce our problem to the following special case (cf. $\q'$, $\q'_{k}$):
\[
\q = \oplus_{i = 1}^{n} (\rho \otimes \nu_{a_{i}} \otimes \nu_{b_{i}}),
\]
where $A_{i} \geqslant A_{i-1}, B_{i} \geqslant B_{i-1}$ and $\zeta_{i} = +$. In this case, we have the following result.

\begin{theorem}% special case THEOREM
\label{thm: special case}

Suppose we are in the special case described above.

\begin{enumerate}

\item There is a bijection
\[
\Pkt{\q} \rightarrow \{(\ul, \ueta) \in \mathbb{Z}^{n} \times \{\pm 1\}^{n} \, | \, 0 \leqslant l_{i} \leqslant [(A_{i} - B_{i} +1)/2],
\text{ \eqref{eq: nonvanishing 1} and \eqref{eq: sign condition} are satisfied }\} / \sim
\]
where $\ul = (l_{i})$, $\ueta = (\eta_{i})$, and
\begin{align}% nonvanishing 1 Eq
\label{eq: nonvanishing 1}
\begin{cases}
\eta_{i+1} = (-1)^{A_{i} - B_{i}} \eta_{i}  & \Rightarrow A_{i+1} - l_{i+1} \geqslant A_{i} - l_{i}, \quad B_{i+1} + l_{i+1} \geqslant B_{i} + l_{i},  \\
\eta_{i+1} \neq (-1)^{A_{i} - B_{i}}\eta_{i}  & \Rightarrow B_{i+1} + l_{i+1} > A_{i} - l_{i}.   
\end{cases}
\end{align}
and
\begin{align}% sign condition Eq
\label{eq: sign condition}
\prod_{i = 1}^{n} \epsilon_{l_{i}, \eta_{i}} = 1
\end{align}
where $\epsilon_{l_{i}, \eta_{i}}$ is defined as in \eqref{eq: sign}. We have identified $(\ul, \ueta) \sim (\ul', \ueta')$, whenever 
\[
\begin{cases}
\ul  = \ul' \\
\eta_{i}  = \eta'_{i} \text{ unless } l_{i} = (A - B + 1)/2 
\end{cases}
\]
\item Let $\pi(\q, \ul, \ueta)$ be the representation parametrized by $(\ul, \ueta)$. Consider the maximal sequence of integers 
\[
0 = k_{0} < \cdots < k_{r} = n
\]
such that $A_{k_{j}} - l_{k_{j}}< B_{k_{j} + 1} + l_{k_{j} + 1}$. When $A_{i} - l_{i} \geqslant B_{i+1} + l_{i+1}$, we take
\[
t_{i} = \frac{(A_{i}  - l_{i}) + (B_{i+1} + l_{i+1})}{2}
\] 
and 
\[
\delta_{i} = \begin{cases}
1 & \text{ if } \, t_{i} - A_{i} \in \mathbb{Z} \\
1/2 & \text{ if } \, t_{i} - A_{i} \notin \mathbb{Z}
\end{cases}
\]
Then we have
\begin{align*}
\pi(\q, \ul, \ueta) & \hookrightarrow \times_{i = 1}^{n} \underbrace{\begin{pmatrix}
              B_{i} & \cdots & -A_{i} \\
              \vdots &  & \vdots \\
              B_{i} + l_{i} - 1 & \cdots & -(A_{i} - l_{i} + 1) \end{pmatrix}}_{I_{i}} \\
& \times_{i = 1}^{r} \times_{k_{i-1} < j < k_{i}} 
\underbrace{\begin{pmatrix}
              B_{j+1} + l_{j+1} & \cdots & -(A_{j} - l_{j}) \\
              \vdots &  & \vdots \\
              t_{j} - \delta_{j} & \cdots & -(t_{j} + \delta_{j}) 
\end{pmatrix}}_{\tilde{I}_{j}}
\rtimes \pi'
\end{align*}
as the unique irreducible subrepresentation, where  
\[
\pi'  = \pi\Big(\cup_{i} \Big\{ \cdots \Big\} \Big)
\] 
with
\[
\Big\{ \cdots \Big\}  = \Big\{ (\rho, A_{k_{i}} - l_{k_{i}}, B_{k_{i}} + l_{k_{i}}, 0, \eta_{k_{i}}, +) \Big\}
\] 
if $k_{i} - k_{i-1} = 1$, and  
\begin{align*}
\Big\{ \cdots \Big\} & = \Bigg\{ \Big(\rho, A_{k_{i}} - l_{k_{i}}, t_{k_{i} - 1} - \delta_{k_{i} - 1} + 1, 0, (-1)^{t_{k_{i} - 1} - \delta_{k_{i} - 1} + 1 - (B_{k_{i}} + l_{k_{i}}) }\eta_{k_{i}}, + \Big), \\
&  \cup_{k_{i -1} + 1 < j < k_{i}} \Big(\rho, t_{j} + \delta_{j} - 1, t_{j-1} - \delta_{j-1} + 1, 0, (-1)^{t_{j-1} - \delta_{j-1} + 1 - (B_{j} + l_{j})}\eta_{j}, + \Big), \\
& \Big(\rho, t_{k_{i-1}+1} + \delta_{k_{i-1} + 1} - 1, B_{k_{i-1}+1} + l_{k_{i-1} + 1}, 0 , \eta_{k_{i-1} + 1}, + \Big) \Bigg\}
\end{align*} 
otherwise. Here $\pi'$ is a tempered representation of a group $G'(F)$ of the same type as $G(F)$, and its complete Langlands parameter $(\phi', \epsilon')$ can be described as in \eqref{eq: discrete series in A-packet}. Moreover, the complete Langlands parameter $(\phi, \epsilon)$ of $\pi(\q, \ul, \ueta)$ is given as follows: 
\[
\phi = (\oplus_{j}  \tilde{\phi}_{j}) \oplus (\oplus_{i}  \phi_{i}) \oplus \phi' \oplus (\oplus_{i}  \phi_{i}^{\vee}) \oplus (\oplus_{j}  \tilde{\phi}^{\vee}_{j})
\]
where $\phi_{i}$ (resp. $\tilde{\phi}_{j}$) is the Langlands parameter of the corresponding shifted Speh representation $I_{i}$ (resp. $\tilde{I}_{j}$) and $\epsilon$ corresponds to $\epsilon'$ under the natural isomorphism
\(
\mathcal{S}_{\phi} \cong \mathcal{S}_{\phi'}.
\)
\end{enumerate}

\end{theorem}

\begin{remark}
When the intervals $[B_{i}, A_{i}]$ are disjoint, the condition \eqref{eq: nonvanishing 1} becomes void. In that case, the result is due to M{\oe}glin \cite[Theorem 4.2]{Moeglin:2009}.
\end{remark}

\subsection{Even orthogonal groups}

Let $G$ be a quasisplit special even orthogonal group over $F$, split over a quadratic extension $E/F$. Let  $\theta_{0}$ be an outer automorphism of $G$ over $F$, induced from the conjugate action of the even orthogonal group. Let $\Sigma_{0} = \langle \theta_{0} \rangle$ and $G^{\Sigma_{0}} = G \rtimes \Sigma_{0}$, which is isomorphic to the even orthogonal group. Let $\hat{\theta}_{0}$ be the dual automorphism on $\hat{G}$, which commutes with the action of ${\rm Gal}(E/F)$. The local Langlands correspondence for $G^{\Sigma_{0}}(F)$ takes the following form: there is a canonical bijection (after fixing a $\theta_{0}$-stable Whittaker datum)  
\[
{\rm Irr}(G^{\Sigma_{0}}(F)) \cong \{(\phi, \epsilon) | \phi \in \bar{\Phi}(G(F)), \epsilon \in {\rm Irr}(\mathcal{S}^{\Sigma_{0}}_{\phi})\},
\]
where $\bar{\Phi}(G(F))$ is the set of $\hat{\theta}_{0}$-orbits of Langlands parameters of $G(F)$, which are $\D{G} \rtimes \langle \hat{\theta}_{0} \rangle$-conjugacy classes of admissible homomorphisms 
\[
\phi: W_{F} \times SL(2, \mathbb{C}) \rightarrow \D{G} \rtimes {\rm Gal}(E/F)
\]
and
\[
\mathcal{S}^{\Sigma_{0}}_{\phi} := \pi_{0}(Z_{\D{G} \rtimes \langle \hat{\theta}_{0} \rangle}(\phi)/Z(\D{G})^{{\rm Gal}(E/F)} ).
\]
This follows from Arthur's results on the local Langlands correspondence for $G(F)$ and the $\theta_{0}$-twisted endoscopic character relations (cf. \cite{Arthur:2013} and \cite[Theorem 4.3]{Xu:Apacket}). We will call the pair $(\phi, \epsilon)$ complete Langlands parameter of $G^{\Sigma_{0}}(F)$, and denote the corresponding representation by $\pi^{\Sigma_{0}}(\phi, \epsilon)$.

For any Arthur parameter $\psi$ of $G(F)$, Arthur \cite{Arthur:2013} has associated it with a finite multi-set $\cPkt{\q}$ of $\Sigma_{0}$-orbits in ${\rm Irr}(G(F))$, in the same way as we have described for symplectic and special odd orthogonal groups. This is also multiplicity free due to Moeglin \cite{Moeglin1:2011}. In \cite{Xu:Comb} we define the Arthur packet $\Pkt{\q}^{\Sigma_{0}}$ for $G^{\Sigma_{0}}(F)$ to be the subset of isomorphism classes of irreducible representations of $G^{\Sigma_{0}}(F)$, whose restriction to $G(F)$ have irreducible constituents in $\cPkt{\q}$. In this paper, we will also prove the analogues of Theorem~\ref{thm: push integral}, ~\ref{thm: push half-integral}, ~\ref{thm: special case} for $\Pkt{\q}^{\Sigma_{0}}$.

\section{Review of Moeglin's parametrization}% review SECTION
\label{sec: review}

From now on, we will let $G$ be a quasisplit symplectic or special orthogonal group over a $p$-adic field $F$. In order to get a uniform description, we will also take $\Sigma_{0} = 1$ and $G^{\Sigma_{0}} = G$, when $G$ is not special even orthogonal. Let $\q$ be an Arthur parameter of $G(F)$. We will review Moeglin's parametrization of elements in $\Pkt{\q}^{\Sigma_{0}}$. The reader is referred to \cite{Xu:Apacket} \cite{Xu:Comb} for more details. 

Let $\q_{p}$ be the parameter consisting of Jordan blocks of $\q$ that has the same parity as $\D{G}$, and $>_{\q}$ be an admissible order on $Jord(\q_{p})$. The admissibility condition requires that for any $(\rho, A, B, \zeta), (\rho, A', B', \zeta') \in Jord(\q_{p})$ satisfying $A > A', B > B'$ and $\zeta = \zeta'$, we have $(\rho, A, B, \zeta) >_{\q} (\rho, A', B', \zeta')$. Then M{\oe}glin showed that there is an injection depending on $>_{\q}$
\begin{align}% parametrization Eq
\label{eq: parametrization}
\Pkt{\q}^{\Sigma_{0}} \hookrightarrow \Big\{(\ul, \ueta) \in \mathbb{Z}^{Jord(\q_{p})} \times \{\pm 1\}^{Jord(\q_{p})} \, | \, \ul(\rho, A, B, \zeta) \in [0, (A - B + 1)/2], \text{ \eqref{eq: sign conditin general} is satisfied} \Big\}/_{\sim_{\Sigma_{0}}},
\end{align}
where 
\begin{align}% sign condition general Eq
\label{eq: sign conditin general}
\prod_{(\rho, A, B, \zeta) \in Jord(\q_{p})} \epsilon_{\ul, \ueta}(\rho, A, B, \zeta) = 1
\end{align}
and 
\[
\epsilon_{\ul, \ueta}(\rho, A, B, \zeta) := \ueta(\rho, A, B, \zeta)^{A - B + 1}(-1)^{[(A - B + 1)/2] + \ul(\rho, A, B, \zeta)}.
\] 
Here we say $(\ul, \ueta) \sim_{\Sigma_{0}} (\ul', \ueta')$ if and only if 
\[
\begin{cases}
\ul = \ul' \\
(\ueta/\ueta')(\rho, A, B, \zeta) = 1 \text{ unless } \ul(\rho, A, B, \zeta) = (A - B + 1)/2.
\end{cases}
\] 
This is the parametrization appearing in Section~\ref{subsec: special case}, where we have implicitly chosen the order $>_{\q}$ to be that of the indexes. For any $(\ul, \ueta)$ in \eqref{eq: parametrization}, we let $\pi^{\Sigma_{0}}_{M, >_{\q}}(\q, \ul, \ueta)$ be the associated representation if $(\ul, \ueta)$ is in the image, or zero otherwise. M{\oe}glin also expressed the nonvanishing of $\pi^{\Sigma_{0}}_{M, >_{\q}}(\q, \ul, \eta)$ in terms of the nonvanishing of certain Jacquet module (cf. \eqref{eq: push back}). Following this description, we have developed a procedure in \cite{Xu:Comb} to determine the image explicitly. As an application, we give the formula \eqref{eq: nonvanishing 1} for characterizing the image in the special case (cf. Section~\ref{subsec: special case}). The proof will be given in Appendix~\ref{sec: nonvanishing}. 

What turns out crucial to this procedure \cite{Xu:Comb} is to understand how the injection \eqref{eq: parametrization} changes when one changes the order $>_{\q}$. This is also one of the main results in \cite{Xu:Comb} and we will recall it here. Suppose we have two adjacent Jordan blocks $(\rho, A_i, B_i, \zeta_{i})$ $(i = 1, 2)$ with respect to the admissible order $>_{\psi}$, and
\[
(\rho, A_2, B_2, \zeta_{2}) >_{\psi} (\rho, A_1, B_1, \zeta_{1}).
\]
Suppose the new order $>'_{\psi}$ obtained by switching the two is still admissible. Then by definition, either $\zeta_{1} \neq \zeta_{2}$ or one of $\{[B_{i}, A_{i}]\}_{i = 1, 2}$ is included in the other. Let us define $\q_{-}$ by
\[
Jord(\q_{-}) = Jord(\q) \backslash \{(\rho, A_{2}, B_{2}, \zeta_{2}), (\rho, A_{1}, B_{1}, \zeta_{1})\}.
\]
Suppose 
\[
\r^{\Sigma_{0}}_{M, >_{\q}}(\q, \ul, \ueta) = \r^{\Sigma_{0}}_{M, >'_{\q}}(\q, \ul', \ueta') \neq 0,
\]
then the restrictions of $(\ul, \ueta)$ and $(\ul', \ueta')$ to $Jord(\q_{-})$ are equivalent with respect to ($\sim_{\Sigma_{0}}$) and the following conditions are satisfied.

\begin{enumerate}

\item If $\zeta_{1} = \zeta_{2}$, it suffices to consider the case $[B_2, A_2] \supseteq [B_1, A_1]$. Then we are in one of the following situations.

\begin{enumerate}

\item If $\eta_{2} \neq (-1)^{A_{1} - B_{1}} \eta_{1}$ and $\eta'_{1} = (-1)^{A_{2} - B_{2}} \eta'_{2}$, then
         \[
         \begin{cases}
         l_{1} = l'_{1} \\
          l_{2} - l'_{2} = (A_{1} - B_{1} - 2l_{1}) + 1 \\
         \eta'_{1} = (-1)^{A_{2} - B_{2}} \eta_{1} 
         \end{cases}
         \]
         
\item If $\eta_{2} = (-1)^{A_{1} - B_{1}} \eta_{1}$ and $\eta'_{1} \neq (-1)^{A_{2} - B_{2}}\eta'_{2}$, then 
         \[
         \begin{cases}
         l_{1} = l'_{1} \\
         l'_{2} - l_{2} = (A_{1} - B_{1} - 2l_{1}) + 1 \\
         \eta'_{1} = (-1)^{A_{2} - B_{2}} \eta_{1} 
         \end{cases}
         \]         
         
\item If $\eta_{2} = (-1)^{A_{1} - B_{1}} \eta_{1}$ and $\eta'_{1} = (-1)^{A_{2} - B_{2}}\eta'_{2}$, then 
         \[
         \begin{cases}
         l_{1} = l'_{1} \\
         (l'_{2} - l'_{1}) + (l_{2} - l_{1}) = (A_{2} - B_{2}) - (A_{1} - B_{1}) \\
         \eta'_{1} = (-1)^{A_{2} - B_{2}} \eta_{1} 
         \end{cases}
         \]          

\end{enumerate}

\item If $\zeta_{1} \neq \zeta_{2}$, then 
\[
\begin{cases}
l'_{2} = l_{2} \\
l'_{1} = l_{1} \\
\eta_{2} = (-1)^{A_{1} - B_{1} + 1} \eta'_{2} \\
\eta_{1} = (-1)^{A_{2} - B_{2} + 1} \eta'_{1}
\end{cases}
\]

\end{enumerate}
This formula suggests that for $(\rho, A, B, \zeta) \in Jord(\q_{p})$ with $B = 0$, the choice of sign $\zeta$ will affect the parametrization (cf. \cite[Proposition 7.5]{Xu:Comb}).

\subsection{Terminology}

We recall a few terminologies from \cite{Xu:Apacket} \cite{Xu:Comb}. Let $\q$ be an Arthur parameter of $G(F)$ such that $\q = \q_{p}$. Let $\rho$ be an irreducible unitary supercuspidal representation of $GL(d_{\rho}, F)$. We denote by $Jord_{\rho}(\q)$ the subset of $Jord(\q)$ containing $\rho$. A subset $J$ of $Jord_{\rho}(\q)$ is said to have {\bf discrete diagonal restriction} if the intervals $[B, A], [B', A']$ do not intersect for any $(\rho, A, B, \zeta), (\rho, A', B', \zeta') \in Jord_{\rho}(\q)$. We say $\q$ has discrete diagonal restriction if $Jord_{\rho}(\q)$ has discrete diagonal restriction for all $\rho$. A Jordan block $(\rho, A, B, \zeta)$ is said to be {\bf far away} from a subset $J$ of $Jord_{\rho}(\q)$ if
\[
B > 2^{|J|} \cdot \Big(\sum_{(\rho, A', B', \zeta') \in J} A' + |J| \sum_{(\rho, A', B', \zeta') \in Jord_{\rho}(\q)}(A' - B' + 1) \Big)
\]
and we will write $(\rho, A, B, \zeta) \gg J$ (cf. \cite[Section 2]{Xu:Comb}). 

Let $>_{\q}$ be an admissible order on $Jord(\q)$ and we index $Jord_{\rho}(\q)$ for each $\rho$ so that 
\[
(\rho, A_{i}, B_{i}, \zeta_{i}) >_{\q} (\rho, A_{i-1}, B_{i-1}, \zeta_{i-1}).
\]
A new parameter $\q_{\gg}$ is said to {\bf dominate} $\q$ with respect to $>_{\q}$ if $Jord_{\rho}(\q_{\gg})$ is obtained by shifting $(\rho, A_{i}, B_{i}, \zeta_{i})$ to $(\rho, A_{i} + T_{i}, B_{i} + T_{i}, \zeta_{i})$ with $T_{i} \geqslant 0$ for each $\rho$, and $>_{\q}$ induces an admissible order on $Jord(\q_{\gg})$. In this case, $\r_{M, >_{\q}}^{\Sigma_{0}}(\q, \ul, \ueta)$ and $\r_{M, >_{\q}}^{\Sigma_{0}}(\q_{\gg}, \ul, \ueta)$ are related as follows:
\begin{align}% push back Eq
\label{eq: push back}
\r_{M, >_{\q}}^{\Sigma_{0}}(\q, \ul, \ueta) := \circ_{\{\rho: Jord_{\rho}(\q) \neq \emptyset \}} \circ_{(\rho, A_{i}, B_{i}, \zeta_{i}) \in Jord_{\rho}(\q)} \Jac_{(\rho, A_{i} + T_{i}, B_{i} + T_{i}, \zeta_{i}) \mapsto (\rho, A_{i}, B_{i}, \zeta_{i})} \r_{M, >_{\q}}^{\Sigma_{0}}(\q_{\gg}, \ul, \ueta),
\end{align}
where $i$ is decreasing (cf. \cite[Remark 8.4]{Xu:Apacket}). If we further assume both of them are nonzero, then we have
\[
\r_{M, >_{\q}}^{\Sigma_{0}}(\q_{\gg}, \ul, \ueta) \hookrightarrow \times_{\{\rho: Jord_{\rho}(\q) \neq \emptyset \}} \times_{(\rho, A_{i}, B_{i}, \zeta_{i}) \in Jord_{\rho}(\q)} \begin{pmatrix}
              \zeta_{i} (B_{i} + T_{i}) & \cdots & \zeta_{i}(B_{i} + 1) \\
              \vdots &  & \vdots \\
              \zeta_{i} (A_{i} + T_{i}) & \cdots & \zeta_{i}(A_{i} + 1)
       \end{pmatrix} \rtimes \r_{M, >_{\q}}^{\Sigma_{0}}(\q, \ul, \ueta).
\]
where $i$ is increasing (cf. \cite[Proposition 8.5]{Xu:Apacket}).

At last, we will say a few words about the operators used in \eqref{eq: push back}. Let $M = GL(d_{\rho}) \times G_{-}$ be the Levi component of a standard maximal parabolic subgroup $P$ of $G$. For any finite-length smooth representation $\r^{\Sigma_{0}}$ of $G^{\Sigma_{0}}(F)$, we can decompose the semisimplification of its Jacquet module as follows
\[
s.s. \, \Jac_{P}(\r^{\Sigma_{0}}) = \bigoplus_{i} \tau_{i} \otimes \sigma_{i},
\] 
where $\tau_{i}$ (resp. $\sigma_{i}$) are irreducible representations of $GL(d_{\rho}, F)$ (resp. $G_{-}^{\Sigma_{0}}(F)$). We define $\Jac_{x}\r^{\Sigma_{0}}$ for any real number $x$ to be 
\begin{align*}
\Jac_{x}(\r^{\Sigma_{0}}) = \bigoplus_{\tau_{i} = \rho||^{x}} \sigma_{i}.
\end{align*}
We also define
\[
\Jac_{x_{1}, \cdots, x_{s}}\r^{\Sigma_{0}} = \Jac_{x_{s}} \circ \cdots \circ \Jac_{x_{1}} \r^{\Sigma_{0}}
\]
for any ordered sequence of real numbers $\{x_{1}, \cdots, x_{s}\}$. Let
\[
X_{i} = \begin{bmatrix}
    \zeta_{i} (B_{i} + T_{i}) & \cdots & \zeta_{i}(B_{i} + 1) \\
              \vdots &  & \vdots \\
              \zeta_{i} (A_{i} + T_{i}) & \cdots & \zeta_{i}(A_{i} + 1)
  \end{bmatrix}
\]
with respect to $(\rho, A_{i} + T_{i}, B_{i} + T_{i}, \zeta_{i})$ in the previous paragraph. Then $\Jac_{(\rho, A_{i} + T_{i}, B_{i} + T_{i}, \zeta_{i}) \mapsto (\rho, A_{i}, B_{i}, \zeta_{i})}$ is defined to be $\Jac_{X_{i}} := \circ_{x \in X_{i}} \Jac_{x}$, where $x$ ranges over $X_{i}$ from top to bottom and left to right.

\section{Step One}% step one SECTION
\label{sec: step one}

In the next three sections, we will give the proofs of the main results stated in the introduction. The Arthur parameters (cf. \eqref{eq: local A-parameter}) considered in these proofs are always under the assumption \eqref{eq: fix rho}. Later we will make some comments on the general case (cf. Section~\ref{sec: general}).

In step one, we consider a subclass of representations in the special case (cf. Section~\ref{subsec: special case}). In the special case, we have
\[
\q = \oplus_{i = 1}^{n} (\rho \otimes \nu_{a_{i}} \otimes \nu_{b_{i}})
\]
where $A_{i} \geqslant A_{i-1}, B_{i} \geqslant B_{i-1}$ and $\zeta_{i} = +$. We fix the order so that
\[
(\rho, A_{i}, B_{i}, \zeta_{i}) >_{\q} (\rho, A_{i-1}, B_{i-1}, \zeta_{i-1}).
\]
Now let us consider $\pi^{\Sigma_{0}}_{M, >_{\psi}}(\q, \ul, \ueta)$, where $\ul = 0$. In this case, we can reinterpret the nonvanishing condition \eqref{eq: nonvanishing 1} as follows.
\begin{itemize}

\item If $A_{i}\geqslant B_{i+1}$, then $\eta_{i+1} = (-1)^{A_{i} - B_{i}}\eta_{i}$. Let $t_{i} = (A_{i} + B_{i+1})/2$ and 
\[
\delta_{i} = \begin{cases}
1 & \text{ if $t_{i} - A_{i} \in \mathbb{Z}$ } \\
1/2 & \text{ if $t_{i} - A_{i} \notin \mathbb{Z}$ }
\end{cases}
\]

\item If $A_{i} < B_{i+1}$, then there is no condition on $\eta_{i+1}$.

\end{itemize}
Consider the maximal sequence of integers 
\[
0 = k_{0} < \cdots < k_{r} = n
\]
such that $A_{k_{j}} < B_{k_{j} + 1}$. We would like to show

\begin{theorem}% step one THEOREM
\label{thm: step one}
\begin{align}% step one
\label{eq: step one}
& \pi^{\Sigma_{0}}_{M, >_{\psi}}(\q, \ul, \ueta) \hookrightarrow \times_{i = 1}^{r} \, \times_{k_{i-1} < j < k_{i}} 
\underbrace{\begin{pmatrix}
              B_{j+1}  & \cdots & -A_{j} \\
              \vdots &  & \vdots \\
              t_{j} - \delta_{j} & \cdots & -(t_{j} + \delta_{j}) 
\end{pmatrix}}_{\tilde{I}_{j}} \rtimes \sigma^{\Sigma_{0}} 
\end{align}
as the unique irreducible subrepresentation, where
\[
\sigma^{\Sigma_{0}} := \pi^{\Sigma_{0}}_{M, >_{\psi}} \Big(\cup_{i} \Big\{  \cdots \Big\} \Big)
\]
is a tempered representation with 
\[
\Big\{ \cdots \Big\} = \Big\{ (\rho, A_{k_{i}}, B_{k_{i}}, 0, \eta_{k_{i}}, +) \Big\}
\]
if $k_{i} - k_{i-1} = 1$, and
\begin{align*}
\Big\{ \cdots \Big\} = & \Bigg\{ \Big(\rho, A_{k_{i}}, t_{k_{i} - 1} - \delta_{k_{i} - 1} + 1, 0, (-1)^{t_{k_{i} - 1} - \delta_{k_{i} - 1} + 1 - B_{k_{i}} }\eta_{k_{i}}, + \Big), \\
& \cup_{k_{i -1} + 1 < j < k_{i}}  \Big(\rho, t_{j} + \delta_{j} - 1, t_{j-1} - \delta_{j-1} + 1, 0, (-1)^{t_{j-1} - \delta_{j-1} + 1 - B_{j} }\eta_{j}, + \Big), \\
& \Big(\rho, t_{k_{i-1}+1} + \delta_{k_{i-1} + 1} - 1, B_{k_{i-1}+1}, 0 , \eta_{k_{i-1} + 1}, + \Big) \Bigg\}
\end{align*}
otherwise. Moreover, the induced representation $\mathcal{I}$ in \eqref{eq: step one} is a subrepresentation of the costandard representation, obtained by taking induction of the shifted Steinberg representations from rows of $\tilde{I}_{j}$ together with $\sigma^{\Sigma_{0}}$.

\end{theorem}

The following corollary is an immediate consequence of the theorem.

\begin{corollary}
In the notations of Theorem~\ref{thm: step one}, the complete Langlands parameter $(\phi, \epsilon)$ of $\pi^{\Sigma_{0}}_{M, >_{\psi}}(\q, \ul, \ueta)$ is given as follows:
\[
\phi = (\oplus_{j}  \tilde{\phi}_{j}) \oplus \phi' \oplus (\oplus_{j} \tilde{\phi}_{j}^{\vee}) 
\]
where $\tilde{\phi}_{j}$ is the Langlands parameter of $\tilde{I}_{j}$, $(\phi', \epsilon')$ is the complete Langlands parameter of $\sigma^{\Sigma_{0}}$, and $\epsilon$ corresponds to $\epsilon'$ under the natural isomorphism
\(
\mathcal{S}^{\Sigma_{0}}_{\phi} \cong \mathcal{S}^{\Sigma_{0}}_{\phi'}.
\)
\end{corollary}

%This expression for $\pi^{\Sigma_{0}}_{M, >_{\psi}}(\q, \ul, \ueta)$ above corresponds to the Langlands classification of nontempered representations. So $\pi^{\Sigma_{0}}_{M, >_{\psi}}(\q, \ul, \ueta)$ is the unique irreducible subrepresentation.
 
We will prove Theorem~\ref{thm: step one} by induction on 
\[
\sum_{i=1}^{n-1} \text{ max }\{ A_{i} - B_{i + 1}, 0\}.
\]
Suppose it is not zero. Let us choose the maximal integer $s < n$ such that $A_{s} - B_{s + 1} \geqslant A_{i} - B_{i+1}$ for all $1 \leqslant i < n$. By maximality of $s$, we have $B_{s + 2} > B_{s +1}$ or $s = n-1$. Moreover, there exists $l \leqslant s + 1$ such that
\[
A_{s} = A_{s - 1} = \cdots = A_{l-1} \text{ and } B_{s+1} = B_{s} = \cdots = B_{l}
\]
and $A_{l-1} > A_{l-2}$ or $l = 2$. 

\begin{lemma}% step one
\label{lemma: step one}
\[
\pi^{\Sigma_{0}}_{M, >_{\psi}}(\q, \ul, \ueta) \hookrightarrow \times_{j = l-1}^{s} \langle B_{j+1}, \cdots, -A_{j} \rangle \rtimes \pi^{\Sigma_{0}}_{M, >_{\psi}}(\q', \ul, \ueta'),
\] 
where $A'_{j} = A_{j} - 1$, $B'_{j+1} = B_{j + 1} + 1$ and $\eta'_{j + 1} = - \eta_{j + 1}$ for $l-1 \leqslant j \leqslant s$.
\end{lemma}

\begin{proof}
Since $\ul = 0$, we can reorganize the Jordan blocks for $l-1 \leqslant j \leqslant s+1$ as
\begin{align*}
& (\rho, A_{s+1}, B_{s+1} + 1, 0, -\eta_{s+1}, +) > (\rho, B_{s+1}, B_{s+1}, 0, \eta_{s+1}, +)  > (\rho, A_{s}, B_{s+1}, 0, \eta_{s}, +) \\
& > \cdots > (\rho, A_{l}, B_{l+1}, 0, \eta_{l}, +) > (\rho, A_{l-1}, B_{l}, 0, (-1)^{B_{l} - B_{l-1}}\eta_{l-1}, +) > (\rho, B_{l} - 1, B_{l-1}, 0, \eta_{l-1}, +),
\end{align*}
where we have splitted the first and last ones. Note the last Jordan block above disappear when $B_{l} = B_{l-1}$. Then we can move $(\rho, B_{s+1}, B_{s+1}, 0, \eta_{s+1}, +)$ to the second last position (or last when $B_l = B_{l-1}$) above. By the change of order formula, it can be combined with the last term. So we get
\begin{align*}
& (\rho, A_{s+1}, B_{s+1} + 1, 0, -\eta_{s+1}, +) >  (\rho, A_{s}, B_{s+1}, 1, -\eta_{s}, +) \\
& > \cdots > (\rho, A_{l}, B_{l+1}, 1, -\eta_{l}, +) > (\rho, A_{l-1}, B_{l}, 1, -(-1)^{B_{l} - B_{l-1}}\eta_{l-1}, +) > (\rho, B_{l}, B_{l-1}, 0, \eta_{l-1}, +)
\end{align*}
Since $A_{l-1} > A_{l-2}$ and $B_{s + 2} > B_{s +1}$, we get by applying Lemma~\ref{lemma: step two}
\[
\pi^{\Sigma_{0}}_{M, >_{\psi}}(\q, \ul, \ueta) \hookrightarrow \times_{j = l-1}^{s} \langle B_{j+1}, \cdots, -A_{j} \rangle \rtimes \pi^{\Sigma_{0}}_{M, >_{\psi}}(\q', \ul, \ueta').
\] 
\end{proof}

By Lemma~\ref{lemma: step one} and the induction assumption, we have
\begin{align}% step one reduction Eq
\label{eq: step one reduction}
& \pi^{\Sigma_{0}}_{M, >_{\psi}}(\q, \ul, \ueta) \hookrightarrow \times_{j = l-1}^{s} \langle B_{j+1}, \cdots, -A_{j} \rangle \rtimes \mathcal{I}'
\end{align}
and
\begin{align*}
\mathcal{I}' := \times_{i = 1}^{r} \, \times_{k_{i-1} < j < k_{i}} 
\underbrace{\begin{pmatrix}
              B'_{j+1}   & \cdots & - A'_{j}  \\
              \vdots &  & \vdots \\
              t_{j} - \delta_{j} & \cdots & -(t_{j} + \delta_{j}) 
\end{pmatrix}}_{\tilde{I}_{j}} \rtimes \, \sigma^{\Sigma_{0}}
\end{align*}
Moreover, $\mathcal{I}'$ is a subrepresentation of the costandard representation, obtained by taking induction of the shifted Steinberg representations from rows of $\tilde{I}_{j}$ together with $\sigma^{\Sigma_{0}}$. Combined with the maximality of $A_{s} - B_{s+1}$, we see the induction in \eqref{eq: step one reduction} is a subrepresentation of the costandard representation as we want. It also follows that the induction in \eqref{eq: step one reduction} has a unique irreducible subrepresentation. Since $\tilde
{I}_{j}$ are interchangeable with each other (cf. \cite[Corollary 4.3]{Xu:Comb}), one can combine $\langle B_{j+1}, \cdots, -A_{j} \rangle$ with $\tilde{I}_{j}$ for $l - 1 \leqslant j \leqslant s$, and this gives \eqref{eq: step one}.

\section{Step Two}% step two SECTION
\label{sec: step two}

We will settle the special case in this step, hence complete the proof of Theorem~\ref{thm: special case}. Let
\[
\q = \oplus_{i = 1}^{n} (\rho \otimes \nu_{a_{i}} \otimes \nu_{b_{i}})
\]
where $A_{i} \geqslant A_{i-1}, B_{i} \geqslant B_{i-1}$ and $\zeta_{i} = +$. We fix the order so that
\[
(\rho, A_{i}, B_{i}, \zeta_{i}) >_{\q} (\rho, A_{i-1}, B_{i-1}, \zeta_{i-1}).
\]
By Theorem~\ref{thm: nonvanishing}, $\pi^{\Sigma_{0}}_{M, >_{\psi}}(\q, \ul, \ueta) \neq 0$ if and only if \eqref{eq: nonvanishing 1} is satisfied. We would like to prove the following theorem.

\begin{theorem}% step two THEOREM
\label{thm: step two}
\begin{align}% step two Eq
\label{eq: step two}
\pi^{\Sigma_{0}}_{M, >_{\psi}}(\q, \ul, \ueta) \hookrightarrow \times_{i} \underbrace{\begin{pmatrix}
              B_{i} & \cdots & -A_{i} \\
              \vdots &  & \vdots \\
              B_{i} + l_{i} - 1 & \cdots & -(A_{i} - l_{i} + 1) \end{pmatrix}}_{I_{i}} \rtimes \pi^{\Sigma_{0}}_{M, >_{\psi}} \Big(\cup_{i}(\rho, A_{i} - l_{i}, B_{i} + l_{i}, 0, \eta_{i}, \zeta_{i})\Big)
\end{align}
as the unique irreducible subrepresentation. Moreover, after applying \eqref{eq: step one} to 
\begin{align}% step two 1 Eq
\label{eq: step two 1}
\pi^{\Sigma_{0}}_{M, >_{\psi}} \Big(\cup_{i}(\rho, A_{i} - l_{i}, B_{i} + l_{i}, 0, \eta_{i}, \zeta_{i})\Big),
\end{align}
we can embed the right hand side of \eqref{eq: step two} into an induced representation $\mathcal{I}$. Then $\mathcal{I}$ is a subrepresentation of the costandard representation, obtained by taking induction of the shifted Steinberg representations from the shifted Speh representations with a tempered representation $\sigma^{\Sigma_{0}}$ as in Theorem~\ref{thm: step one}.
\end{theorem}

The following corollary is an immediate consequence of the theorem.

\begin{corollary}
In the notations of Theorem~\ref{thm: step two}, the complete Langlands parameter $(\phi, \epsilon)$ of $\pi^{\Sigma_{0}}_{M, >_{\psi}}(\q, \ul, \ueta)$ is given as follows:
\[
\phi = (\oplus_{i} \phi_{i}) \oplus \phi' \oplus (\oplus_{i}  \phi_{i}^{\vee}) 
\]
where $\phi_{i}$ is the Langlands parameter of $I_{i}$, $(\phi', \epsilon')$ is the complete Langlands parameter of \eqref{eq: step two 1}, and $\epsilon$ corresponds to $\epsilon'$ under the natural isomorphism
\(
\mathcal{S}^{\Sigma_{0}}_{\phi} \cong \mathcal{S}^{\Sigma_{0}}_{\phi'}.
\)
\end{corollary}

We will prove Theorem~\ref{thm: step two} by induction on 
\(
\sum_{i=1}l_{i}.
\)
Among all $i$ such that $l_{i} \neq 0$, let us choose maximal $s$ for the property that $A_{s} - B_{s} \geqslant A_{i} - B_{i}$ for any such $i$. By the maximality of $s$, we have $B_{s + 1} > B_{s}$ or $s = n$. Moreover, there exists $l \leqslant s$ such that
\[
A_{s} = \cdots = A_{l} \text{ and } B_{s} = \cdots = B_{l}
\]
and $A_{l} > A_{l-1}$ or $l = 1$.

\begin{lemma}% step two LEMMA
\label{lemma: step two}
\[
\pi^{\Sigma_{0}}_{M, >_{\psi}}(\q, \ul, \ueta) \hookrightarrow \times_{i = l}^{s} \langle B_{i}, \cdots, -A_{i} \rangle \rtimes \pi^{\Sigma_{0}}_{M, >_{\psi}}(\q', \ul', \ueta),
\] 
where $A'_{i} = A_{i} - 1$, $B'_{i} = B_{i} + 1$, and $l'_{i} = l_{i} - 1$ for $l \leqslant i \leqslant s$.
\end{lemma}

\begin{proof}
Let $\q_{\gg}^{(l)}$ be a dominating parameter of $\q$ such that the Jordan blocks for $i<l$ remains the same, and the Jordan blocks for $i \geqslant l$ are shifted by $T_{i}$, so that they are disjoint and far away from $\cup_{i<l}\{(\rho, A_{i}, B_{i}, \zeta_{i})\}$. Similarly, we can define $\q_{\gg}^{(s+1)}$ (resp. $\q_{\gg}^{'(s+1)}$).
\begin{align*}
\pi^{\Sigma_{0}}_{M, >_{\psi}}(\q_{\gg}^{(l)}, \ul, \ueta) \hookrightarrow & \times_{i = l}^{s} \langle B_{i} + T_{i}, \cdots, -(A_{i} + T_{i}) \rangle \times \\
&  \times_{i = l}^{s} \begin{pmatrix}
              B_{i} +T_{i} + 1 & \cdots & B_{i} + 2 \\
              \vdots &  & \vdots \\
              A_{i} + T_{i} - 1 & \cdots & A_{i} \end{pmatrix} \rtimes \pi^{\Sigma_{0}}_{M, >_{\psi}}(\q_{\gg}^{'(s+1)}, \ul', \ueta) \\
\hookrightarrow &  \times_{i = l}^{s} \Big(\langle B_{i} + T_{i}, \cdots, -A_{i} \rangle \times \underbrace{\langle -(A_{i} + 1), \cdots, -(A_{i} + T_{i}) \rangle}_{I_{i}} \Big) \times \\
&  \times_{i = l}^{s} \underbrace{\begin{pmatrix}
              B_{i} +T_{i} + 1 & \cdots & B_{i} + 2 \\
              \vdots &  & \vdots \\
              A_{i} + T_{i} - 1 & \cdots & A_{i} \end{pmatrix}}_{II_{i}} \rtimes \, \pi^{\Sigma_{0}}_{M, >_{\psi}}(\q_{\gg}^{'(s+1)}, \ul', \ueta) 
\end{align*}
We can switch $I_{i}$ with $II_{j}$ (cf. \cite[Corollary 4.3]{Xu:Comb}). Since $A_{l} > A_{l-1}$, we can then take the dual of $I_{i}$ (cf. \cite[Proposition 4.6]{Xu:Comb}). Moreover, we can combine $II_{i}$ with $I^{\vee}_{i}$, for otherwise, ${\rm Jac}_{A_{i} + T_{i}} \, \pi^{\Sigma_{0}}_{M, >_{\psi}}(\q_{\gg}^{(l)}, \ul, \ueta) \neq 0$, which is impossible. Therefore, we get
\begin{align*}
\pi^{\Sigma_{0}}_{M, >_{\psi}}(\q_{\gg}^{(l)}, \ul, \ueta) \hookrightarrow &  \times_{i = l}^{s} \langle B_{i} + T_{i}, \cdots, -A_{i} \rangle \times \\
&  \times_{i = l}^{s} \begin{pmatrix}
              B_{i} +T_{i} + 1 & \cdots & B_{i} + 2 \\
              \vdots &  & \vdots \\
              A_{i} + T_{i}  & \cdots & A_{i} + 1 \end{pmatrix} \rtimes \pi^{\Sigma_{0}}_{M, >_{\psi}}(\q_{\gg}^{'(s+1)}, \ul', \ueta) \\
\hookrightarrow & \times_{i = l}^{s} \Big( \langle B_{i} + T_{i}, \cdots, -A_{i} \rangle \times  \begin{pmatrix}
              B_{i} +T_{i} + 1 & \cdots & B_{i} + 2 \\
              \vdots &  & \vdots \\
              A_{i} + T_{i}  & \cdots & A_{i} + 1 \end{pmatrix}\Big) \rtimes \pi^{\Sigma_{0}}_{M, >_{\psi}}(\q_{\gg}^{'(s+1)}, \ul', \ueta)
\end{align*}
It follows
\begin{align*}
\pi^{\Sigma_{0}}_{M, >_{\psi}}(\q_{\gg}^{(s+1)}, \ul, \ueta) \hookrightarrow &  \times_{i = l}^{s} \langle B_{i}, \cdots, -A_{i} \rangle \rtimes \pi^{\Sigma_{0}}_{M, >_{\psi}}(\q_{\gg}^{'(s+1)}, \ul', \ueta) \\ \hookrightarrow &  \times_{i = l}^{s} \langle B_{i}, \cdots, -A_{i} \rangle \times \Big( \times_{i > s} \underbrace{\begin{pmatrix}
              B_{i} +T_{i} & \cdots & B_{i} + 1 \\
              \vdots &  & \vdots \\
              A_{i} + T_{i}  & \cdots & A_{i} + 1 \end{pmatrix}}_{III_{i}} \Big) \rtimes \pi^{\Sigma_{0}}_{M, >_{\psi}}(\q', \ul', \ueta).
\end{align*}
Since $B_{s+1} > B_{s}$, we can switch $\langle B_{i}, \cdots, -A_{i} \rangle$ with $III_{j}$ (cf. \cite[Corollary 4.3]{Xu:Comb}). Hence, 
\[
\pi^{\Sigma_{0}}_{M, >_{\psi}}(\q, \ul, \ueta) \hookrightarrow \times_{i = l}^{s} \langle B_{i}, \cdots, -A_{i} \rangle \rtimes \pi^{\Sigma_{0}}_{M, >_{\psi}}(\q', \ul', \ueta).
\]

\end{proof}

By Lemma~\ref{lemma: step two} and the induction assumption, we have
\begin{align}% step two reduction Eq
\label{eq: step two reduction}
\pi^{\Sigma_{0}}_{M, >_{\psi}}(\q, \ul, \ueta) \hookrightarrow & \times_{i = l}^{s} \langle B_{i}, \cdots, -A_{i} \rangle \rtimes \mathcal{I}'
\end{align}
and
\begin{align*}
\mathcal{I}' := & \times_{i = l}^{s} \underbrace{\begin{pmatrix}
              B_{i} + 1 & \cdots & -(A_{i} - 1) \\
              \vdots &  & \vdots \\
              B_{i} + l_{i} - 1 & \cdots & -(A_{i} - l_{i} + 1) \end{pmatrix}}_{II_{i}} \times 
\times_{i < l \text{ or } i> s} \underbrace{\begin{pmatrix}
              B_{i} & \cdots & -A_{i} \\
              \vdots &  & \vdots \\
              B_{i} + l_{i} - 1 & \cdots & -(A_{i} - l_{i} + 1) \end{pmatrix}}_{II_{i}}  \\
& \times \times_{\text{ some } j} \underbrace{\begin{pmatrix}
              B_{j+1} + l_{j+1} & \cdots & -(A_{j} - l_{j}) \\
              \vdots &  & \vdots \\
              t_{j} - \delta_{j} & \cdots & -(t_{j} + \delta_{j}) 
\end{pmatrix}}_{\tilde{I}_{j}}
\rtimes \, \sigma^{\Sigma_{0}}
\end{align*}
Moreover, $\mathcal{I}'$ is a subrepresentation of the costandard representation, obtained by taking induction of the shifted Steinberg representations from the shifted Speh representations with $\sigma^{\Sigma_{0}}$. We claim the induction in \eqref{eq: step two reduction} is a subrepresentation of the costandard representation as we want. 

To prove the claim, we need to show any shifted Steinberg representation above, whose shift is less than that of $\langle B_{s}, \cdots, -A_{s} \rangle$, can be moved to the front. By our choice of $s$, it suffices to consider $\langle x, \cdots, -y \rangle$ from rows of $\tilde{I}_{j}$. Moreover, it is necessary that $l_{j+1} = l_{j} = 0$. There are two cases.
\begin{enumerate}

\item If $A_{s} \leqslant A_{j}$, then $y \geqslant A_{s}$. 

\item If $B_{s} \geqslant B_{j+1}$, then $x \leqslant B_{s}$. 

\end{enumerate}
In either case, we see $\langle x, \cdots, -y \rangle$ and $\langle B_{s}, \cdots, -A_{s} \rangle$ are interchangeable (cf. \cite[Corollary 4.3]{Xu:Comb}). This finishes the proof of our claim. As a consequence, the induction in \eqref{eq: step two reduction} has a unique irreducible subrepresentation. So we can combine $\langle B_{i}, \cdots, -A_{i} \rangle$ with $II_{i}$ for $l \leqslant i \leqslant s$, and this gives \eqref{eq: step two}.

\section{Step Three}% step three SECTION
\label{sec: step three}

In this step we will prove Theorem~\ref{thm: push integral} and Theorem~\ref{thm: push half-integral}, which reduce our problem to the special case settled in the previous step. In order to apply the induction argument, we need to generalize our problem (cf. \eqref{eq: fix b}) to the following case 
\[
\q = \oplus_{i = 1}^{n} (\rho \otimes \nu_{a_{i}} \otimes \nu_{b_{i}})
\]
where $A_{i} \geqslant A_{i-1}$ and there exists $m \leqslant n$ such that

\begin{itemize}

\item if $i > m$, then $\zeta_{i} = +$, $B_{i+1} \geqslant B_{i}$;

\item if $i \leqslant m$, then $\zeta_{i} = -$, and $B_{i} \leqslant B_{i-1}$.

\end{itemize}
We choose the order so that
\[
(\rho, A_{i}, B_{i}, \zeta_{i}) >_{\q} (\rho, A_{i-1}, B_{i-1}, \zeta_{i-1}).
\]
Among all $i \leqslant m$ such that $B_{i} \neq 0$ (resp. $1/2$), we choose $s \leqslant m$ maximal for the property that 
\begin{itemize}
\item $A_{s} \geqslant A_{i}$ for all such $i$;

\item $B_{s} \geqslant B_{i}$, if $A_{s} = A_{i}$ for any such $i$
\end{itemize}
By the maximality of $s$, we have $B_{s} > B_{s + 1}$ or $s = m$. Moreover, there exists $l \leqslant s$ such that
\[
A_{s} = \cdots = A_{l} \text{ and } B_{s} = \cdots = B_{l}
\]
and $A_{l} > A_{l-1}$ or $l = 1$.

\begin{lemma}% step three LEMMA
\label{lemma: step three}
There is a bijection between
\begin{align*}
\Pkt{\q}^{\Sigma_{0}} & \rightarrow \Pkt{\q^{s}}^{\Sigma_{0}} \\
\pi^{\Sigma_{0}}_{M, >_{\psi}}(\psi, \ul, \ueta) & \mapsto \pi^{\Sigma_{0}}_{M, >_{\psi}}(\psi^{s}, \ul, \ueta)
\end{align*}
such that 
\[
\pi^{\Sigma_{0}}_{M, >_{\psi}}(\psi, \ul, \ueta) \hookrightarrow \times_{i = l}^{s} \langle -B_{i}, \cdots, -A_{i} \rangle \rtimes \pi^{\Sigma_{0}}_{M, >_{\psi}}(\psi^{s}, \ul, \ueta),
\]
where $\psi^{s}$ is obtained from $\psi$ by changing $(\rho, A_{i}, B_{i}, \zeta_{i})$ to $(\rho, A_{i} - 1, B_{i} - 1, \zeta_{i})$ for $l \leqslant i \leqslant s$.
\end{lemma}

\begin{proof}
We first show  
\[
\pi^{\Sigma_{0}}_{M, >_{\psi}}(\psi, \ul, \ueta) \neq 0 \Leftrightarrow \pi^{\Sigma_{0}}_{M, >_{\psi}}(\psi^{s}, \ul, \ueta) \neq 0.
\] 
Following the procedure in \cite[Section 8]{Xu:Comb}, we can first reduce to the case that all $(\rho, A_{i}, B_{i}, \zeta_{i})$ for $i > m$ are far away from $\cup_{i \leqslant m}\{(\rho, A_{i}, B_{i}, \zeta_{i})\}$, except for one $(\rho, A_{j}, B_{j}, \zeta_{j})$. This is done by the operations of ``pull" and ``expand" (cf. \cite[Section 7.1, 7.2]{Xu:Comb}). Since $A_{j} \geqslant A_{i}$ for all $i \leqslant m$, then we can ``expand" $[B_{j}, A_{j}]$ and change the sign $\zeta_{j}$ to negative (cf. \cite[Section 7.3]{Xu:Comb}). In this way, we can further reduce to the case that all $(\rho, A_{i}, B_{i}, \zeta_{i})$ for $i > m$ are far away from $\cup_{i \leqslant m}\{(\rho, A_{i}, B_{i}, \zeta_{i})\}$. Since $A_{l} > A_{l-1}$ and $B_{s} > B_{s+1}$, the inclusion relations of intervals are not changed after shifting $[B_{i}, A_{i}]$ to $[B_{i} - 1, A_{i} - 1]$ for $l \leqslant i \leqslant s$. Then it is not hard to see from the procedure in \cite[Section 8]{Xu:Comb} again that the nonvanishing condition is not changed.

Next we impose a new order $>'_{\psi}$ by moving $\{ (\rho, A_{i}, B_{i}, \zeta_{i}) \}_{i = l}^{s}$ to the front. Suppose  
\[
\pi^{\Sigma_{0}}_{M, >_{\psi}}(\psi, \ul, \ueta) = \pi^{\Sigma_{0}}_{M, >'_{\psi}}(\psi, \ul', \ueta'),
\]
then 
\[
\pi^{\Sigma_{0}}_{M, >_{\psi}}(\psi^{s}, \ul, \ueta) = \pi^{\Sigma_{0}}_{M, >'_{\psi}}(\psi^{s}, \ul', \ueta')
\]
by the change of order formula. So it suffices to prove the lemma under this new order. Let $\psi_{\gg}$ be the parameter obtained by shifting $[B_{i}, A_{i}]$ to $[B_{i} + T_{i}, A_{i} + T_{i}]$ for $l \leqslant i \leqslant s$, which are disjoint and far away from the rest. Then
\begin{align*}
\pi^{\Sigma_{0}}_{M, >'_{\psi}}(\psi_{\gg}, \ul', \ueta')  \hookrightarrow  \times_{i = l}^{s} \begin{pmatrix}
              -(B_{i} + T_{i}) & \cdots &  -(A_{i} + T_{i}) \\
              \vdots &  & \vdots \\
              -B_{i} & \cdots & -A_{i} \end{pmatrix} \rtimes \pi^{\Sigma_{0}}_{M, >'_{\psi}}(\psi^{s}, \ul', \ueta'),
\end{align*}
where $i$ increases. It follows 
\begin{align*}
\pi^{\Sigma_{0}}_{M, >'_{\psi}}(\psi_{\gg}, \ul', \ueta') & \hookrightarrow \begin{pmatrix}
              -(B_{l} + T_{l}) & \cdots & -(A_{l} + T_{l}) \\
              \vdots &  & \vdots \\
              -(B_{l} + 1) & \cdots & -(A_{l} + 1) \end{pmatrix} \times
\langle -B_{l}, \cdots, -A_{l} \rangle \times \\
& \times_{i = l+1}^{s} \begin{pmatrix}
              -(B_{i} + T_{i}) & \cdots & -(A_{i} + T_{i}) \\
              \vdots &  & \vdots \\
               -B_{i} & \cdots & -A_{i} \end{pmatrix} \rtimes \pi^{\Sigma_{0}}_{M, >'_{\q}}(\psi^{s}, \ul', \ueta') \\
& \hookrightarrow \begin{pmatrix}
              -(B_{l} + T_{l}) & \cdots & -(A_{l} + T_{l}) \\
              \vdots &  & \vdots \\
               -(B_{l} + 1) & \cdots & -(A_{l} + 1) \end{pmatrix} \times \\
& \times_{i = l+1}^{s} \begin{pmatrix}
              -(B_{i} + T_{i}) & \cdots &  -(A_{i} + T_{i}) \\
              \vdots &  & \vdots \\
               -B_{i} & \cdots & -A_{i} \end{pmatrix} \times 
\langle -B_{l}, \cdots, -A_{l} \rangle \rtimes \pi^{\Sigma_{0}}_{M, >'_{\psi}}(\psi^{s}, \ul', \ueta')
\end{align*}
Continuing this way, we should get
\begin{align*}
\pi^{\Sigma_{0}}_{M, >'_{\psi}}(\psi_{\gg}, \ul', \ueta') \hookrightarrow & \times_{i=l}^{s} \begin{pmatrix}
              -(B_{i} + T_{i}) & \cdots & -(A_{i} + T_{i}) \\
              \vdots &  & \vdots \\
              -(B_{i} + 1) & \cdots & -(A_{i} + 1) \end{pmatrix} \times \\
& \times_{i = l}^{s}
\langle -B_{i}, \cdots, -A_{i} \rangle \rtimes\pi^{\Sigma_{0}}_{M, >'_{\psi}}(\psi^{s}, \ul', \ueta')
\end{align*}
It follows
\[
\pi^{\Sigma_{0}}_{M, >'_{\psi}}(\psi, \ul', \ueta') \hookrightarrow \times_{i = l}^{s} \langle -B_{i}, \cdots, -A_{i} \rangle \rtimes \pi^{\Sigma_{0}}_{M, >'_{\psi}}(\psi^{s}, \ul', \ueta').
\]
This finishes the proof.

\end{proof}

\begin{remark}
In this lemma, it is critical to have $A_{j} \geqslant A_{i}$ for $\zeta_{j} = +, \zeta_{i} = -$. Here we give a counter-example when this condition is not satisfied. Suppose
\[
Jord(\psi) = \{(\rho, A_{1}, B_{1}, \zeta_{1}), (\rho, A_{2}, B_{2}, \zeta_{2})\}
\]
with $A_{i}, B_{i} \in \mathbb{Z}$, $\zeta_{1} = -$, $\zeta_{2} = +$, and
\(
A_{1} > B_{1} + B_{2} > A_{2} > A_{1} - B_{1}.
\)
Let 
\[
Jord(\psi') = \{(\rho, A_{1} - B_{1}, 0, \zeta_{1}), (\rho, A_{2}, B_{2}, \zeta_{2})\}.
\] 
We claim $|\Pkt{\psi}^{\Sigma_{0}}| > |\Pkt{\psi'}^{\Sigma_{0}}|$, which will result in a contradiction to Lemma~\ref{lemma: step three}. After changing $\zeta_{1}$ to positive, we can apply Theorem~\ref{thm: special case} to $\psi'$. Then there is a bijection
\[
\Pkt{\psi'}^{\Sigma_{0}} \rightarrow \{(\ul, \ueta) \in \mathbb{Z}^{2} \times \{\pm 1\}^{2} \, | \, 0 \leqslant l_{i} \leqslant [(A_{i} - B_{i} +1)/2],
\text{ \eqref{eq: nonvanishing 1} and \eqref{eq: sign condition} are satisfied }\} / \sim_{\Sigma_{0}}
\]
where $\ul = (l_{i})$, $\ueta = (\eta_{i})$. By our assumption, $[0, A_{1} - B_{1}]$ intersects with $[A_{2}, B_{2}]$, so the condition \eqref{eq: nonvanishing 1} is not void. On the other hand, let
\[
Jord(\psi'') = \{(\rho, A_{1}, B_{1}, \zeta_{1}), (\rho, A_{2} - B_{2}, 0, \zeta_{2})\}
\] 
By applying $\Jac_{(\rho, A_{2}, B_{2}, \zeta_{2}) \mapsto (\rho, A_{2} - B_{2}, 0, \zeta_{2})}$, we get a surjection
\begin{align}% counterexample
\label{eq: counterexample}
\Pkt{\psi}^{\Sigma_{0}} \rightarrow \Pkt{\psi''}^{\Sigma_{0}}
\end{align}
By our assumption, $[0, A_{2} - B_{2}]$ does not intersect with $[A_{1}, B_{1}]$. So by \cite[Theorem 4.2]{Moeglin:2009}, there is a bijection
\[
\Pkt{\psi''}^{\Sigma_{0}} \rightarrow \{(\ul, \ueta) \in \mathbb{Z}^{2} \times \{\pm 1\}^{2} \, | \, 0 \leqslant l_{i} \leqslant [(A_{i} - B_{i} +1)/2],
\text{ \eqref{eq: sign condition} is satisfied }\} / \sim_{\Sigma_{0}}
\]
Hence, $|\Pkt{\psi}^{\Sigma_{0}}| \geqslant |\Pkt{\psi''}^{\Sigma_{0}}| > |\Pkt{\psi'}^{\Sigma_{0}}|$. From \eqref{eq: parametrization}, we also see \eqref{eq: counterexample} is actually a bijection.

\end{remark}

\subsection{Integral case}

We assume $A_{i}, B_{i} \in \mathbb{Z}$. Recall $\zeta_{i} = +$ for $i > m$ and $\zeta_{i} = -$ for $i \leqslant m$. We get a new parameter $\q'$ by replacing all $(\rho, A_{i}, B_{i}, \zeta_{i})$ by $(\rho, A'_{i}, B'_{i}, \zeta'_{i})$ such that: $\zeta'_{i} = \zeta_{i}$ and 
\begin{itemize}

\item $A'_{i} = A_{i}, B'_{i} = B_{i}$ for $i > m$;

\item $A'_{i} = A_{i} - B_{i}, B'_{i} = 0$ for $i \leqslant m$.

\end{itemize}

\begin{theorem}% step three integral THEOREM
\label{thm: step three integral}
There is a bijection 
\begin{align*}
\Pkt{\q}^{\Sigma_{0}} & \rightarrow \Pkt{\q'}^{\Sigma_{0}} \\
\pi^{\Sigma_{0}}_{M, >_{\psi}}(\psi, \ul, \ueta) & \mapsto \pi^{\Sigma_{0}}_{M, >_{\psi}}(\psi', \ul, \ueta)
\end{align*}
such that 
\begin{align}% step three integral Eq
\label{eq: step three integral}
\pi^{\Sigma_{0}}_{M, >_{\psi}}(\psi, \ul, \ueta) \hookrightarrow \times_{i \leqslant m} \underbrace{\begin{pmatrix}
              -B_{i} & \cdots & -A_{i} \\
              \vdots &  & \vdots \\
              -1 & \cdots & -(A_{i} - B_{i} +1)
       \end{pmatrix}}_{I_{i}} \rtimes \pi^{\Sigma_{0}}_{M, >_{\psi}}(\psi', \ul, \ueta)
\end{align}
as the unique irreducible subrepresentation. 

Let us assume 
\[
\pi^{\Sigma_{0}}_{M, >_{\psi}}(\psi', \ul, \ueta) = \pi^{\Sigma_{0}}_{M, >_{\psi}}(\psi', \ul', \ueta').
\]
after changing $\zeta_{i}$ to positive for $i \leqslant m$. By applying \eqref{eq: step two} and \eqref{eq: step one} to $\pi^{\Sigma_{0}}_{M, >_{\psi}}(\psi', \ul', \ueta')$, we can embed the right hand side of \eqref{eq: step three integral} into an induced representation $\mathcal{I}$. Then $\mathcal{I}$ is a subrepresentation of the costandard representation, obtained by taking induction of the shifted Steinberg representations from the shifted Speh representations with a tempered representation $\sigma^{\Sigma_{0}}$ as in Theorem~\ref{thm: step one}.
\end{theorem}

\begin{proof}

We can prove this by induction on $\sum_{i \leqslant m} B_{i}$. By Lemma~\ref{lemma: step three} and the induction assumption, we have
\begin{align}% step three integral reduction Eq
\label{eq: step three integral reduction}
\pi^{\Sigma_{0}}_{M, >_{\q}}(\psi, \ul, \ueta) \hookrightarrow & \times_{i = l}^{s} \langle -B_{i}, \cdots, -A_{i} \rangle \rtimes \mathcal{I}^{s}
\end{align}
and
\begin{align*}
\mathcal{I}^{s} := & \times_{i=l}^{s} \underbrace{\begin{pmatrix}
              -(B_{i} -1 ) & \cdots & -(A_{i} -1) \\
              \vdots &  & \vdots \\
               -1 & \cdots & -(A_{i} - B_{i} + 1) \end{pmatrix}}_{III_{i}} \times \times_{i < l \text{ or } m \geqslant i > s} \underbrace{\begin{pmatrix}
              -B_{i} & \cdots & -A_{i} \\
              \vdots &  & \vdots \\
               -1 & \cdots & -(A_{i} - B_{i} + 1) \end{pmatrix}}_{III_{i}} \times \\
& \times_{i = 1}^{n} \underbrace{\begin{pmatrix}
              B'_{i} & \cdots & -A'_{i}\\
              \vdots &  & \vdots \\
              B'_{i} + l'_{i} - 1 & \cdots & -(A'_{i} - l'_{i} + 1) \end{pmatrix}}_{II_{i}}
\times \times_{\text{ some } j} \underbrace{\begin{pmatrix}
              B'_{j+1} +l'_{j+1} & \cdots & -(A'_{j} - l'_{j}) \\
              \vdots &  & \vdots \\
              t'_{j} - \delta_{j} & \cdots & -(t'_{j} + \delta_{j}) \end{pmatrix}}_{\tilde{I}_{j}} \rtimes \, \sigma^{\Sigma_{0}}
\end{align*}
Moreover, $\mathcal{I}^{s}$ is a subrepresentation of the costandard representation, obtained by taking induction of the shifted Steinberg representations from the shifted Speh representations with $\sigma^{\Sigma_{0}}$. We claim the induced representation in \eqref{eq: step three integral reduction} is a subrepresentation of the costandard representation as we want. 

To prove the claim, we need to show any shifted Steinberg representation above, whose shift is less than that of $\langle -B_{s}, \cdots, -A_{s} \rangle$, can be moved to the front. There are two cases.
\begin{enumerate}
\item If it is in the form $\langle -x, \cdots, -y \rangle$ from $III_{i}$, then by our choice of $s$, we have $A_{s} \geqslant y$ and $x \geqslant B_{s}$. Hence, $\langle -x, \cdots, -y \rangle$ and $\langle -B_{s}, \cdots, -A_{s} \rangle$ are interchangeable.

\item If it is in the form $\langle x, \cdots, -y \rangle$ from $II_{i}$ or $\tilde{I}_{j}$, then we have $y \geqslant A_{s}$. Otherwise,
\[
\frac{x - y}{2} \geqslant -\frac{y}{2} > - \frac{A_{s}}{2} \geqslant - \frac{A_{s} + B_{s}}{2},
\]
which contradicts to our assumption about the shifts. Hence, $\langle x, \cdots, -y \rangle$ and $\langle -B_{s}, \cdots, -A_{s} \rangle$ are interchangeable (cf. \cite[Corollary 4.3]{Xu:Comb}).
\end{enumerate}
This finishes the proof of our claim. As a consequence, the induction in \eqref{eq: step three integral reduction} has a unique irreducible subrepresentation. So we can combine $\langle -B_{i}, \cdots, -A_{i} \rangle$ with $III_{i}$ for $l \leqslant i \leqslant s$, and this gives \eqref{eq: step three integral}.

\end{proof}

The following corollary is an immediate consequence of Theorem~\ref{thm: step three integral}.

\begin{corollary}
In the notations of Theorem~\ref{thm: step three integral}, the complete Langlands parameter $(\phi, \epsilon)$ of $\pi^{\Sigma_{0}}_{M, >_{\psi}}(\q, \ul, \ueta)$ is given as follows:
\[
\phi = (\oplus_{i}  \phi_{i}) \oplus \phi' \oplus (\oplus_{i}  \phi_{i}^{\vee}) 
\]
where $\phi_{i}$ is the Langlands parameter of $I_{i}$, $(\phi', \epsilon')$ is the complete Langlands parameter of $\pi^{\Sigma_{0}}_{M, >_{\psi}}(\psi', \ul, \ueta)$, and $\epsilon$ corresponds to $\epsilon'$ under the natural isomorphism
\(
\mathcal{S}^{\Sigma_{0}}_{\phi} \cong \mathcal{S}^{\Sigma_{0}}_{\phi'}.
\)
\end{corollary}

\subsection{Half-integral case}

We assume $A_{i}, B_{i} \notin \mathbb{Z}$. Recall $\zeta_{i} = +$ for $i > m$ and $\zeta_{i} = -$ for $i \leqslant m$. 

\begin{theorem}% step three half-integral THEOREM
\label{thm: step three half-integral}
Consider the maximal sequence of integers
\[
0 = s_{0} < s_{1} < \cdots < s_{l} = m
\]
such that $A_{s_{j}} - B_{s_{j}} \neq A_{s_{j} + 1} - B_{s_{j} + 1}$. For any $0 \leqslant k \leqslant l$, we get a new parameter $\q'_{k}$ by replacing all $(\rho, A_{i}, B_{i}, \zeta_{i})$ by $(\rho, A'_{i}, B'_{i}, \zeta'_{i})$ such that: $\zeta'_{i} = +$ and
\begin{itemize}

\item $A'_{i} = A_{i}, B'_{i} = B_{i}$ for $i > m$;

\item $A'_{i} = A_{i} - B_{i} -1/2, B'_{i} = 1/2$ for $i \leqslant m$ and $i \neq s_{k}$;

\item $A'_{i} = A_{i} - B_{i} + 1/2, B'_{i} = 1/2$ for $i = s_{k}$.

\end{itemize}
Then we can divide $\Pkt{\q}^{\Sigma_{0}}$ into $l + 1$ classes, i.e.,
\[
\Pkt{\q}^{\Sigma_{0}} = \sqcup_{k = 0}^{l} \, \Pkt{\q}^{\Sigma_{0}}(k).
\]
For any $0 \leqslant k \leqslant l$, we put an order $>_{\psi'_{k}}$ on $Jord(\q'_{k})$ so that
\[
(\rho, A'_{i}, B'_{i}, \zeta'_{i}) >_{\psi'_{k}} (\rho, A'_{i-1}, B'_{i-1}, \zeta'_{i-1}).
\]
Then we can get an injection 
\[
\Pkt{\q}^{\Sigma_{0}}(k) \hookrightarrow \Pkt{\q'_{k}}^{\Sigma_{0}}, \quad \pi^{\Sigma_{0}} \mapsto \pi^{\Sigma_{0}}_{M, >_{\q'_{k}}}(\q'_{k}, \ul', \ueta'),
\]
such that
\begin{align}% step three half-integral Eq
\label{eq: step three half-integral}
\pi^{\Sigma_{0}} \hookrightarrow \times_{s_{k} \neq i \leqslant m} \underbrace{\begin{pmatrix}
              -B_{i} & \cdots & -A_{i}\\
              \vdots &  & \vdots \\
              -1/2  & \cdots & -(A_{i} - B_{i} +1/2) \end{pmatrix}}_{I_{i}} \times \underbrace{\begin{pmatrix}
              -B_{s_{k}} & \cdots & -A_{s_{k}} \\
              \vdots &  & \vdots \\
              -3/2 & \cdots & -(A_{s_{k}} - B_{s_{k}} +3/2)
       \end{pmatrix}}_{I_{s_{k}}} \rtimes \pi^{\Sigma_{0}}_{M, >_{\q'_{k}}}(\q'_{k}, \ul', \ueta')
\end{align}
as the unique irreducible subrepresentation. The image is characterized by the condition that for all $i \leqslant s_{k}$,
\begin{itemize}

\item $l'_{i} = 0$;

\item $\eta'_{i} = - \prod_{j < i} (-1)^{A_{j} - B_{j} + 1}$.

\end{itemize}
When $k \neq 0$, the second condition can also be simplified as $\eta'_{1} = -1$. 

After applying \eqref{eq: step two} and \eqref{eq: step one} to $\pi^{\Sigma_{0}}_{M, >_{\q'_{k}}}(\q'_{k}, \ul', \ueta')$, we can embed the right hand side of \eqref{eq: step three half-integral} into an induced representation $\mathcal{I}$. Then $\mathcal{I}$ is a subrepresentation of the costandard representation, obtained by taking induction of the shifted Steinberg representations from the shifted Speh representations with $\sigma^{\Sigma_{0}}$ as in Theorem~\ref{thm: step one}.
\end{theorem}

The following corollary is an immediate consequence of the theorem.

\begin{corollary}
In the notations of Theorem~\ref{thm: step three half-integral}, the complete Langlands parameter $(\phi, \epsilon)$ of $\pi^{\Sigma_{0}}$ is given as follows
\[
\phi = (\oplus_{i}  \phi_{i}) \oplus \phi' \oplus (\oplus_{i} \phi_{i}^{\vee}) 
\]
where $\phi_{i}$ is the Langlands parameter of $I_{i}$, $(\phi', \epsilon')$ is the complete Langlands parameter of $\pi^{\Sigma_{0}}_{M, >_{\q'_{k}}}(\q'_{k}, \ul', \ueta')$, and $\epsilon$ corresponds to $\epsilon'$ under the natural isomorphism
\(
\mathcal{S}^{\Sigma_{0}}_{\phi} \cong \mathcal{S}^{\Sigma_{0}}_{\phi'}.
\)
\end{corollary}

We will prove Theorem~\ref{thm: step three half-integral} by induction on $\sum_{i \leqslant m} (B_{i} - 1/2)$. 

\subsubsection{Change sign}
\label{subsubsection: change sign}

Suppose 
\begin{align}% change sign assumption Eq
\label{eq: change sign assumption}
\sum_{i \leqslant m} (B_{i} - 1/2) = 0,
\end{align}
i.e., $B_{i} = 1/2$ for $i \leqslant m$. We change the order $>_{\q}$ so that
\[
(\rho, A_{i}, 1/2, \zeta_{i}) >_{\q} (\rho, A_{i+1}, 1/2, \zeta_{i+1})
\] 
for $1 \leqslant i \leqslant m-1$. Then there are two cases.

%---------------------------------------------------------------------------------
{\bf First case:} $l_{m} \neq 0$ or $\eta_{m} = 1$. There exists $l \leqslant m$ such that 
\[
A_{m} = A_{m-1} = \cdots = A_{l} \text{ with } A_{l} > A_{l-1} \text{ or } l = 1.
\]
Let $\q''$ be obtained from $\q$ by replacing all $(\rho, A_{i}, B_{i}, \zeta_{i})$ by $(\rho, A_{i} - 1, 1/2, -\zeta_{i})$ for $l \leqslant i \leqslant m$. Note $\psi''$ also falls into the general case that we consider in the beginning of Section~\ref{sec: step three}. We also change $>_{\psi''}$ by reversing the order for the Jordan blocks with negative $\zeta$. It can be obtained by moving $(\rho, A_{i} - 1, 1/2, -\zeta_{i})$ to the front of the last $m$ Jordan blocks, one by one as $i$ goes from $l$ to $m$. In particular, it satisfies
\[
(\rho, A_{i}, B_{i}, \zeta_{i}) >_{\psi''}  (\rho, A_{m} - 1, 1/2, -\zeta_{m}) >_{\psi''} \cdots >_{\psi''} (\rho, A_{l} - 1, 1/2, -\zeta_{l}) >_{\psi''} (\rho, A_{j}, 1/2, \zeta_{j})
\]
for any $i > m$ and $j < l$.

%in particular the number of Jordan blocks $(\rho, A_{i}, B_{i})$ with negative $\zeta_{i}$ reduces from $m$ to $l-1$.

\begin{lemma}% change sign LEMMA
\label{lemma: change sign}
There is a bijection
\[
\Big\{\pi^{\Sigma_{0}}_{M, >_{\psi}}(\psi, \ul, \ueta) \in \Pkt{\q}^{\Sigma_{0}} \, | \, l_{m} \neq 0 \,\, {\rm or} \,\,\eta_{m} = 1 \Big\} \rightarrow \Pkt{\q''}^{\Sigma_{0}}
\]
such that 
\[
\pi^{\Sigma_{0}}_{M, >_{\psi}}(\psi, \ul, \ueta) \hookrightarrow \times_{i = l}^{m} \langle -1/2, \cdots, -A_{i} \rangle \rtimes \pi^{\Sigma_{0}}_{M, >_{\psi''}}(\psi'', \ul'', \ueta''),
\]
where $(\ul'', \ueta'')$ only differs from $(\ul, \ueta)$ for $1 \leqslant i \leqslant m$. We will set $\eta_{m} = -1$ if $l_{m} = (A_{m} + \frac{1}{2})/2$. Then
\begin{align*}
(i < l) \quad & \begin{cases}
l''_{i}  = l_{i} \\
\eta''_{i}  = \eta_{i} (-1)^{(m- l + 1)(A_{i} - \frac{1}{2})} 
\end{cases} \\
(l \leqslant i < m) \quad & \begin{cases}
\eta''_{l} = -\eta_{m} (-1)^{(l-1)(A_{i} + \frac{1}{2})} \\
\eta''_{i+1} = (-1)^{A_{i} - \frac{3}{2}} \eta''_{i} 
\end{cases} \\
(l \leqslant i \leqslant m) \quad l''_{i} = & \begin{cases} l_{m} - 1  & \text{ if } \eta_{m} = -1 \\
l_{m} & \text{ if } \eta_{m} = 1
\end{cases} 
\end{align*}

\end{lemma}

\begin{proof}
If $\pi^{\Sigma_{0}}_{M, >_{\psi}}(\psi, \ul, \ueta) \neq 0$, then it is necessary that 
\[
l_{i} = l_{m} \, \text{ and } \, \eta_{i} = (-1)^{A_{i+1} - \frac{1}{2}} \eta_{i+1}  \quad \text{ for } l \leqslant i < m.
\]
(cf. \cite[Lemma 5.5]{Xu:Comb}) Similarly, if $\pi^{\Sigma_{0}}_{M, >_{\psi''}}(\psi'', \ul'', \ueta'') \neq 0$, then it is necessary that
\[
l''_{i} = l''_{m} \, \text{ and } \, \eta''_{i + 1} = (-1)^{A_{i} - \frac{3}{2}} \eta''_{i}  \quad \text{ for } l \leqslant i < m.
\] 
So for the bijection, it suffices to show 
\[
\pi^{\Sigma_{0}}_{M, >_{\psi}}(\psi, \ul, \ueta) \neq 0 \Leftrightarrow \pi^{\Sigma_{0}}_{M, >_{\psi''}}(\psi'', \ul'', \ueta'') \neq 0
\]
Let $\psi_{\gg}$ (resp. $\psi''_{\gg}$) be dominating parameters obtained from $\psi$ (resp. $\psi''$) by changing $(\rho, A_{i}, B_{i}, \zeta_{i})$ to $(\rho, A_{i} + T_{i}, B_{i} + T_{i}, \zeta_{i})$ for $i > m$, far away from the remaining ones. By Proposition~\ref{prop: change sign} together with the change of order formulas, we know 
\[
\pi^{\Sigma_{0}}_{M, >_{\psi}}(\psi_{\gg}, \ul, \ueta) \neq 0 \Leftrightarrow \pi^{\Sigma_{0}}_{M, >_{\psi''}}(\psi''_{\gg}, \ul'', \ueta'') \neq 0.
\] 
Moreover, we have
\begin{align}% change sign reduction Eq
\label{eq: change sign reduction}
\pi^{\Sigma_{0}}_{M, >_{\psi}}(\psi_{\gg}, \ul, \ueta) \hookrightarrow \times_{i = l}^{m} \langle -1/2, \cdots, -A_{i} \rangle \rtimes \pi^{\Sigma_{0}}_{M, >_{\psi''}}(\psi''_{\gg}, \ul'', \ueta'').
\end{align}
Suppose $\pi^{\Sigma_{0}}_{M, >_{\psi}}(\psi, \ul, \ueta) \neq 0$, then if we apply 
\(
\circ_{i > m} \Jac_{X_{i}}
\)
($i$ decreasing) with 
\[
X_{i} = \begin{bmatrix}
             B_{i} + T_{i} & \cdots & B_{i} + 1 \\
              \vdots &  & \vdots \\
              A_{i} + T_{i} & \cdots & A_{i} + 1
       \end{bmatrix} 
\]
to the right hand side of \eqref{eq: change sign reduction}, we should also get something nonzero. We claim the result must be
\[
\times_{i = l}^{m} \langle -1/2, \cdots, -A_{i} \rangle \rtimes \circ_{i > m} \, \Jac_{X_{i}} \, \pi^{\Sigma_{0}}_{M, >_{\psi''}}(\psi''_{\gg}, \ul'', \ueta''),
\]
which shows the nonvanishing of $\pi^{\Sigma_{0}}_{M, >_{\q''}}(\psi'', \ul'', \ueta'')$. Otherwise, since $A_{i} + 1 > A_{m}$ for $i > m$, there exist $j > m$ and $t \geqslant 1$ such that 
\begin{align*}
\Jac_{X_{j, t}} \circ \, \circ_{j > i > m} \, \Jac_{X_{i}} ({\rm RHS.} \eqref{eq: change sign reduction}) 
\end{align*}
contains a term
\[
\times_{i = l}^{m - 1} \langle -1/2, \cdots, -A_{i} \rangle \times \langle -1/2, \cdots, -(A_{m} - 1) \rangle \times \Jac_{X^{-}_{j, t}} \circ \, \circ_{j > i > m} \, \Jac_{X_{i}}  \, \pi^{\Sigma_{0}}_{M, >_{\psi''}}(\psi''_{\gg}, \ul'', \ueta'') \neq 0
\]
where
\[
X_{j, t} := \begin{bmatrix}
             B_{j} + T_{j} & \cdots & B_{j} + t \\
              \vdots &  & \vdots \\
              A_{j} + T_{j} & \cdots & A_{j} + t
       \end{bmatrix} 
\]
and $X^{-}_{j, t}$ means that we take away the entry $A_{m}$ from the last column of $X_{j, t}$. For $X^{-}_{j, t}$ to be well-defined, we must have
\[
B_{j} + 1 \leqslant B_{j} + t \leqslant A_{m} < A_{j} + 1 \leqslant A_{j} + t
\]
Let $\psi''_{>}$ be obtained from $\psi''_{\gg}$ by changing $(\rho, A_{i} + T_{i}, B_{i} + T_{i}, \zeta_{i})$ back to $(\rho, A_{i}, B_{i}, \zeta_{i})$ for $j > i > m$ and $(\rho, A_{j} + T_{j}, B_{j} + T_{j}, \zeta_{i})$ to $(\rho, A_{j} + t, B_{j} + t, \zeta_{j})$. Then we can conclude
\[
\Jac_{B_{j} + t, \cdots, A_{m} - 1, A_{m} + 1, \cdots, A_{j} + t} \, \pi^{\Sigma_{0}}_{M, >_{\psi''}}(\psi''_{>}, \ul'', \ueta'') \neq 0.
\]
It follows $\Jac_{A_{m} + 1} \pi^{\Sigma_{0}}_{M, >_{\psi''}}(\psi''_{>}, \ul'', \ueta'') \neq 0$. But this is impossible, since $B_{i} + T_{i} > A_{m} + 1$ for $i > j$, and 
\[
B_{i} \leqslant B_{j} < B_{j} + t \leqslant A_{m} < A_{m} + 1
\]
for $j > i > m$. 

Conversely, if $\pi^{\Sigma_{0}}_{M, >_{\psi''}}(\psi'', \ul'', \ueta'') \neq 0$, then 
\[
\pi^{\Sigma_{0}}_{M, >_{\psi''}}(\psi''_{\gg}, \ul'', \ueta'') \hookrightarrow \times_{i > m} \, \mathcal{C}_{X_{i}} \rtimes \pi^{\Sigma_{0}}_{M, >_{\psi''}}(\psi'', \ul'', \ueta''),
\]
where
\[
\mathcal{C}_{X_{i}} := \begin{pmatrix}
             B_{i} + T_{i} & \cdots & B_{i} + 1 \\
              \vdots &  & \vdots \\
              A_{i} + T_{i} & \cdots & A_{i} + 1
       \end{pmatrix}.
\]
Since $\mathcal{C}_{X_{i}}$ and $\langle -1/2, \cdots, -A_{m} \rangle$ are interchangeable for $i > m$, it follows from \eqref{eq: change sign reduction} that
\[
\pi^{\Sigma_{0}}_{M, >_{\psi}}(\psi, \ul, \ueta) \hookrightarrow \times_{i = l}^{m} \langle -1/2, \cdots, -A_{i} \rangle \rtimes \pi^{\Sigma_{0}}_{M, >_{\psi''}}(\psi'', \ul'', \ueta'').
\]
Hence $\pi^{\Sigma_{0}}_{M, >_{\psi}}(\psi, \ul, \ueta) \neq 0$. In the meantime, we have also shown the inclusion relation as well.

\end{proof}

%------------------------------------------------------------------

{\bf Second case:} $l_{m} = 0$ and $\eta_{m} = -1$. By \cite[Lemma 5.5]{Xu:Comb}, it is necessary that $l_{i} = 0$ for $i < m$. Therefore,
\[
\eta_{i} = (-1)^{A_{i+1} - \frac{1}{2}}\eta_{i+1} \quad \text{ for $i < m$. }
\]
We get a new parameter $\psi'$ by changing $(\rho, A_{m}, 1/2, \zeta_{m})$ to $(\rho, A_{m}, 1/2, -\zeta_{m})$, and $(\rho, A_{i}, 1/2, \zeta_{i})$ to $(\rho, A_{i}-1, 1/2, -\zeta_{i})$ for $i<m$. After imposing the usual order $>_{\psi'}$ on the Jordan blocks of $\psi'$, i.e.,
\[
(\rho, A_{m}, 1/2, -\zeta_{m}) >_{\psi'} (\rho, A_{m-1}-1, 1/2, -\zeta_{i}) >_{\psi'} \cdots >_{\psi'} (\rho, A_{1}-1, 1/2, -\zeta_{1}),
\]
we would get $(\ul', \ueta')$ from $(\ul, \ueta)$ by requiring $l'_{i} = 0$ for $i \leqslant m$ and
\[
\eta'_{1} = -1 \text{ and } \eta'_{i+1} = (-1)^{A_{i} - \frac{3}{2}}\eta'_{i} \quad \text{ for $i < m$. }
\]

\begin{lemma}
\[
\pi^{\Sigma_{0}}_{M, >_{\psi}}(\psi, \ul, \ueta) \hookrightarrow \times_{i< m} \langle -1/2, \cdots, -A_{i} \rangle \rtimes \pi^{\Sigma_{0}}_{M, >_{\psi'}}(\psi', \ul', \ueta'),
\]
\end{lemma}

\begin{proof}

Let $\psi^{(k)}$ be the parameter by changing $(\rho, A_{m}, 1/2, \zeta_{m})$ to $(\rho, A_{m}, 1/2, -\zeta_{m})$, and $(\rho, A_{i}, 1/2, \zeta_{i})$ to $(\rho, A_{i}-1, 1/2, -\zeta_{i})$ for $m-k < i < m$. We will also change the order to $>_{\psi^{(k)}}$ by moving these Jordan blocks to the front of the last $m$ Jordan blocks as $i$ goes from $m-k+1$ to $m$. The representations inside the corresponding packets are parametrized by $(\ul^{(k)}, \ueta^{(k)})$ with respect to $>_{\psi^{(k)}}$.

Since $l_{m} = 0$ and $\eta_{m} = -1$, 
\[
\pi^{\Sigma_{0}}_{M, >_{\psi}}(\psi, \ul, \ueta) = \pi^{\Sigma_{0}}_{M, >_{\psi^{(1)}}}(\psi^{(1)}, \ul^{(1)}, \ueta^{(1)}) 
\]
where $(\ul^{(1)}, \ueta^{(1)})$ satisfies
\begin{align*}
\eta^{(1)}_{i} = (-1)^{A_{m} + \frac{1}{2}}\eta_{i} \quad \text{ for $i < m$. }
\end{align*}
Since $l^{(1)}_{i} = 0$ for $i < m$, we also have
\[
\eta^{(1)}_{i} = (-1)^{A_{i+1} - \frac{1}{2}}\eta^{(1)}_{i+1} \quad \text{ for $i < m-1$. }
\]
We compute
\[
\eta^{(1)}_{m-1} = (-1)^{A_{m} + \frac{1}{2}}\eta_{m-1} = (-1)^{A_{m} + \frac{1}{2}} (-1)^{A_{m} - \frac{1}{2}} \eta_{m} = 1.
\]
So we can apply Lemma~\ref{lemma: change sign}. Choose $l \leqslant m-1$ such that 
\[
A_{m-1} = \cdots = A_{l} \text{ with } A_{l} > A_{l-1} \text{ or } l = 1.
\]
Then
\[
\pi^{\Sigma_{0}}_{M, >_{\psi^{(1)}}}(\psi^{(1)}, \ul^{(1)}, \ueta^{(1)}) \hookrightarrow \times_{i = l}^{m-1} \langle -1/2, \cdots, -A_{i} \rangle \rtimes \pi^{\Sigma_{0}}_{M, >_{\psi^{(m - l + 1)}}}(\psi^{(m - l + 1)}, \ul^{(m - l + 1)}, \ueta^{(m - l + 1)}),
\]
To go further, we need to compute 
\begin{align*}
\eta^{(m - l + 1)}_{l-1} = \eta^{(1)}_{l-1} \prod_{i=l}^{m-1}(-1)^{A_{i} - \frac{1}{2}} = \prod_{i = l}^{m-1}(-1)^{A_{i} - \frac{1}{2}} \prod_{i=l}^{m-1}(-1)^{A_{i} - \frac{1}{2}} = 1
\end{align*}
This means we can repeat the previous process. It is not hard to see that one gets eventually
\[
\pi^{\Sigma_{0}}_{M, >_{\psi^{(m - l + 1)}}}(\psi^{(m - l + 1)}, \ul^{(m - l + 1)}, \ueta^{(m - l + 1)}) \hookrightarrow \times_{i< l} \langle -1/2, \cdots, -A_{i} \rangle \rtimes \pi^{\Sigma_{0}}_{M, >_{\psi^{(m)}}}(\psi^{(m)}, \ul^{(m)}, \ueta^{(m)})
\]
As a result, we get
\[
\pi^{\Sigma_{0}}_{M, >_{\psi}}(\psi, \ul, \ueta) \hookrightarrow \times_{i< m} \langle -1/2, \cdots, -A_{i} \rangle \rtimes \pi^{\Sigma_{0}}_{M, >_{\psi^{(m)}}}(\psi^{(m)}, \ul^{(m)}, \ueta^{(m)})
\]
By definition $\psi^{(m)} = \psi'$. So it remains to show $(\ul^{(m)}, \ueta^{(m)}) = (\ul', \ueta')$. It is clear that $l^{(m)}_{i} = 0$ for $i \leqslant m$. Hence,
\[
\eta^{(m)}_{i+1}  = (-1)^{A_{i} - \frac{3}{2}} \eta_{i}^{(m)} \quad \text{ for $i < m$. }
\]
By Lemma~\ref{lemma: change sign}, 
\[
\eta^{(m)}_{1} = -\eta^{(m-k)}_{k}
\]
for maximal $k < m$ such that $A_{1} = A_{i}$ for all $i \leqslant k$. From the above discussion, we see $\eta^{(m-k)}_{k} = 1$. So $\eta^{(m)}_{1} = -1$. Hence, $\ueta^{(m)} = \ueta'$.

\end{proof}

Combining the two cases, we can describe $\Pkt{\q}^{\Sigma_{0}}$ under the assumption \eqref{eq: change sign assumption} as follows. Consider the maximal sequence of integers
\[
0 = s_{0} < s_{1} < \cdots < s_{l} = m
\]
such that $A_{s_{j}} \neq A_{s_{j} + 1}$. For any $0 \leqslant k \leqslant l$, we get a new parameter $\q'_{k}$ by replacing all $(\rho, A_{i}, B_{i}, \zeta_{i})$ by $(\rho, A'_{i}, B'_{i}, \zeta'_{i})$ such that: $\zeta'_{i} = +$ and
\begin{itemize}

\item $A'_{i} = A_{i}, B'_{i} = B_{i}$ for $i > m$;

\item $A'_{i} = A_{i} - 1, B'_{i} = 1/2$ for $i \leqslant m$ and $i \neq s_{k}$;

\item $A'_{i} = A_{i}, B'_{i} = 1/2$ for $i = s_{k}$.

\end{itemize}
Then we can divide $\Pkt{\q}^{\Sigma_{0}}$ into $l + 1$ classes, i.e.,
\[
\Pkt{\q}^{\Sigma_{0}} = \sqcup_{k = 0}^{l} \, \Pkt{\q}^{\Sigma_{0}}(k).
\]
For any $0 \leqslant k \leqslant l$, we can get an injection 
\[
\Pkt{\q}^{\Sigma_{0}}(k) \hookrightarrow \Pkt{\q'_{k}}^{\Sigma_{0}}, \quad \pi^{\Sigma_{0}} \mapsto \pi^{\Sigma_{0}}_{M, >_{\q'_{k}}}(\q'_{k}, \ul', \ueta'),
\]
such that 
\begin{align}% change all sign Eq
\label{eq: change all sign}
\pi^{\Sigma_{0}} \hookrightarrow \times_{ i \leqslant m \text{ and } i \neq s_{k}}  \langle -1/2, \cdots, -A_{i} \rangle \rtimes \pi^{\Sigma_{0}}_{M, >_{\q'_{k}}}(\q'_{k}, \ul', \ueta').
\end{align}
The image is characterized by the condition that for all $i \leqslant s_{k}$,
\begin{itemize}

\item $l'_{i} = 0$;

\item $\eta'_{i} = - \prod_{j < i} (-1)^{A_{j} + \frac{1}{2}}$.

\end{itemize}
Because of the first condition, the second condition can be simplified as $\eta'_{1} = -1$ when $k \neq 0$. 

After applying \eqref{eq: step two} and \eqref{eq: step one} to $\pi^{\Sigma_{0}}_{M, >_{\q'_{k}}}(\q'_{k}, \ul', \ueta')$, we get 
\begin{align}% change all sign reduction Eq
\label{eq: change all sign reduction}
\pi^{\Sigma_{0}} \hookrightarrow \times_{ i \leqslant m \text{ and } i \neq s_{k}}  \langle -1/2, \cdots, -A_{i} \rangle \rtimes \mathcal{I}'
\end{align}
and
\begin{align}% change all sign reduction 1 Eq
\label{eq: change all sign reduction 1}
\mathcal{I}' := \times_{i = 1}^{n} \underbrace{\begin{pmatrix}
              B'_{i} & \cdots & -A'_{i}\\
              \vdots &  & \vdots \\
              B'_{i} + l'_{i} - 1 & \cdots & -(A'_{i} - l'_{i} + 1) \end{pmatrix}}_{II_{i}}
\times \times_{\text{ some } j} \underbrace{\begin{pmatrix}
              B'_{j+1} +l'_{j+1} & \cdots & -(A'_{j} - l'_{j}) \\
              \vdots &  & \vdots \\
              t'_{j} - \delta_{j} & \cdots & -(t'_{j} + \delta_{j}) \end{pmatrix}}_{\tilde{I}_{j}} \rtimes \, \sigma^{\Sigma_{0}}
\end{align}
Note if $\langle x, \cdots, -y \rangle$ from $II_{i}$ or $\tilde{I}_{j}$ has shift less than that of $\langle -1/2, \cdots, -A_{i}\rangle$, then it is necessary that $y \geqslant A_{i}$. So they are interchangeable (cf. \cite[Corollary 4.3]{Xu:Comb}). This shows the induced representation in \eqref{eq: change all sign reduction} is a subrepresentation of the costandard representation as we want. As a consequence, the induced representation in \eqref{eq: change all sign reduction} has a unique irreducible subrepresentation. Therefore, the same is true for that of \eqref{eq: change all sign}.

\subsubsection{Resolution}

Now we can complete the proof of Theorem~\ref{thm: step three half-integral}. By Lemma~\ref{lemma: step three} and induction assumption, we have
\begin{align}% step three half-integral reduction Eq
\label{eq: step three half-integral reduction}
\pi^{\Sigma_{0}} \hookrightarrow & \times_{i = l}^{s} \langle -B_{i}, \cdots, -A_{i} \rangle \rtimes \mathcal{I}^{s}
\end{align}
and
\begin{align*}
\mathcal{I}^{s} :=  & \times_{i = l}^{s} \underbrace{\begin{pmatrix}
              -(B_{i} - 1) & \cdots & -(A_{i} - 1) \\
              \vdots &  & \vdots \\
              -1/2 & \cdots & -(A_{i} - B_{i} +1/2) \end{pmatrix}}_{III_{i}} \times 
\times_{\{i < l \text{ or } m \geqslant i > s\} \backslash \{ s_{k} \}} 
\underbrace{\begin{pmatrix}
              -B_{i} & \cdots & -A_{i}\\
              \vdots &  & \vdots \\
               -1/2 & \cdots & -(A_{i} - B_{i} +1/2) \end{pmatrix}}_{III_{i}} \\
& \times \underbrace{\begin{pmatrix}
              -B_{s_{k}} & \cdots & -A_{s_{k}}\\
              \vdots &  & \vdots \\
               -3/2 & \cdots & -(A_{s_{k}} - B_{s_{k}} +3/2)
       \end{pmatrix}}_{III_{s_{k}}} \rtimes \, \mathcal{I}'
\end{align*}
if $s \neq s_{k}$, and
\begin{align*}
\mathcal{I}^{s} := & \times_{i = l}^{s-1} \underbrace{\begin{pmatrix}
              -(B_{i} - 1) & \cdots & -(A_{i} - 1) \\
              \vdots &  & \vdots \\
               -1/2 & \cdots & -(A_{i} - B_{i} +1/2) \end{pmatrix}}_{III_{i}} \times 
\times_{i < l \text{ or } m \geqslant i > s} 
\underbrace{\begin{pmatrix}
              -B_{i} & \cdots & -A_{i}\\
              \vdots &  & \vdots \\
              -1/2 & \cdots & -(A_{i} - B_{i} +1/2) \end{pmatrix}}_{III_{i}} \\
& \times \underbrace{\begin{pmatrix}
              -(B_{s_{k}} - 1) & \cdots & -(A_{s_{k}} - 1)\\
              \vdots &  & \vdots \\
               -3/2 & \cdots & -(A_{s_{k}} - B_{s_{k}} +3/2)
       \end{pmatrix}}_{III_{s_{k}}} \rtimes \, \mathcal{I}'
\end{align*}
if $s = s_{k}$. Here $\mathcal{I}'$ is defined as in \eqref{eq: change all sign reduction 1}. Moreover, $\mathcal{I}^{s}$ is a subrepresentation of the costandard representation, obtained by taking induction of the shifted Steinberg representations from the shifted Speh representations with $\sigma^{\Sigma_{0}}$. We claim the induced representation in \eqref{eq: step three half-integral reduction} is a subrepresentation of the costandard representation as we want. 

To prove the claim, we need to show any shifted Steinberg representation above, whose shift is less than that of $\langle -B_{s}, \cdots, -A_{s} \rangle$, can be moved to the front. There are two cases.
\begin{enumerate}
\item If it is in the form $\langle -x, \cdots, -y \rangle$ from $III_{i}$, then by our choice of $s$, 
\[
\begin{cases} x \geqslant B_{s} & \text{ if } y \leqslant A_{s}, \\
x = 1/2 < B_{s} & \text{ if } y > A_{s}. \end{cases}
\] 
In either case, $\langle -x, \cdots, -y \rangle$ and $\langle -B_{s}, \cdots, -A_{s} \rangle$ are interchangeable.

\item If it is in the form $\langle x, \cdots, -y \rangle$ from $II_{i}$ or $\tilde{I}_{j}$ in \eqref{eq: change all sign reduction 1}, then we have $y \geqslant A_{s}$. Hence, $\langle x, \cdots, -y \rangle$ and $\langle -B_{s}, \cdots, -A_{s} \rangle$ are interchangeable (cf. \cite[Corollary 4.3]{Xu:Comb}).
\end{enumerate}
This finishes the proof of our claim. As a consequence, the induction in \eqref{eq: step three half-integral reduction} has a unique irreducible subrepresentation. So we can combine $\langle B_{i}, \cdots, -A_{i} \rangle$ with $III_{i}$ for $l \leqslant i \leqslant s$, and this gives \eqref{eq: step three half-integral}.

\section{Comments on the general case}% general SECTION
\label{sec: general}

Let $\q$ be an Arthur parameter of $G(F)$ (cf. \eqref{eq: local A-parameter}) with the assumption that all $b_{i} = b$. Note we do not assume \eqref{eq: fix rho} here. Let $\q_{np}$ be any representation of $W_{F} \times SL(2, \mathbb{C}) \times SL(2, \mathbb{C})$ such that 
\[
\q = \q_{np} \oplus \q_{p} \oplus \q_{np}^{\vee},
\]
where $\q_{np}^{\vee}$ is the dual of $\q_{np}$. M{\oe}glin \cite[Theorem 6]{Moeglin1:2006} proved that there is a bijection
\[
\Pkt{\q}^{\Sigma_{0}} \rightarrow \Pkt{\q_{p}}^{\Sigma_{0}}, \quad \pi^{\Sigma_{0}} \mapsto \tau^{\Sigma_{0}}.
\]
such that
\[
\pi^{\Sigma_{0}} = \Big( \times_{(\rho_{i}, a_{i}, b_{i}) \in Jord(\q_{np})} Sp(St(\rho_{i}, a_{i}), b_{i}) \Big) \rtimes \tau^{\Sigma_{0}}
\]
We can embed $\tau^{\Sigma_{0}}$ into a costandard representation, which is an induction of shifted Steinberg representations and a tempered representation of a group of the same type as $G(F)$. Note these shifted Steinberg representations are interchangeable with the ones from $Sp(St(\rho_{i}, a_{i}), b_{i})$ for $(\rho_{i}, a_{i}, b_{i}) \in Jord(\q_{np})$ by the parity condition. So the complete Langlands parameter $(\phi, \epsilon)$ of $\pi^{\Sigma_{0}}$ will be given as
\[
\phi = (\oplus_{i} \phi_{i}) \oplus \phi' \oplus (\oplus_{i} \phi^{\vee}_{i})
\]
where $\phi_{i}$ is the Langlands parameter of $Sp(St(\rho_{i}, a_{i}), b_{i})$ for $(\rho_{i}, a_{i}, b_{i}) \in Jord(\q_{np})$, $(\phi', \epsilon')$ is the complete Langlands parameter of $\tau^{\Sigma_{0}}$ and $\epsilon$ corresponds to $\epsilon'$ under the canonical isomorphism $\mathcal{S}^{\Sigma_{0}}_{\phi} \cong \mathcal{S}^{\Sigma_{0}}_{\phi'}$. 

At last, we can extend our main results (Theorem~\ref{thm: push integral}, ~\ref{thm: push half-integral}, ~\ref{thm: special case}) to $\q_{p}$ by applying them to each $\rho$ appearing in $Jord(\q_{p})$. For the proofs, it suffices to modify the induction assumptions in the proofs of Theorem~\ref{thm: step one}, ~\ref{thm: step two}, ~\ref{thm: step three integral}, ~\ref{thm: step three half-integral} by considering all Jordan blocks of $\q_{p}$, and apply Theorem~\ref{thm: nonvanishing 2} for the nonvanishing result in the special case (cf. Section~\ref{sec: step one}, ~\ref{sec: step two}).

%------------------------------------------------------------------------------------------------------

\appendix

\section{A nonvanishing result}% nonvanishing SECTION
\label{sec: nonvanishing}

In this appendix, we will prove the following nonvanishing result. Let $\q$ be an Arthur parameter of $G(F)$ (cf. \eqref{eq: local A-parameter}) under the assumption \eqref{eq: fix rho}. Let $>_{\q}$ be an admissible order and we index the Jordan blocks in $Jord(\q)$ such that 
\[
(\rho, A_{i+1}, B_{i+1}, \zeta_{i+1}) >_{\q} (\rho, A_{i}, B_{i}, \zeta_{i}).
\]
Let
\[
J := \cup_{i = 1}^{n}\{(\rho, A_{i}, B_{i}, \zeta_{i})\} \subseteq Jord(\psi)
\] 
Suppose  
\[
A_{i+1} \geqslant A_{i}, \quad B_{i+1} \geqslant B_{i}, \quad \zeta_{i+1} = \zeta_{i} \quad \text{ for } i < n
\]
and
\[
J^{c} \gg J, \quad J^{c} \text{ has discrete diagonal restriction},
\]
where $J^{c} := Jord(\q) \backslash J$. Then we have the following theorem.

\begin{theorem}% nonvanishing THEOREM
\label{thm: nonvanishing}
$\r^{\Sigma_{0}}_{M, >_{\psi}}(\psi, \ul, \ueta) \neq 0$ if and only if the following condition are satisfied for all $i < n$:
\begin{align}% nonvanishing Eq
\label{eq: nonvanishing}
\begin{cases}
\eta_{i+1} = (-1)^{A_{i} - B_{i}}\eta_{i}       & \Rightarrow A_{i+1} - l_{i+1} \geqslant A_{i} - l_{i}, \quad B_{i+1} + l_{i+1} \geqslant B_{i} + l_{i},  \\
\eta_{i+1} \neq (-1)^{A_{i} - B_{i}}\eta_{i}  & \Rightarrow B_{i+1} + l_{i+1} > A_{i} - l_{i}
\end{cases}  
\end{align}
\end{theorem}

\begin{proof}
The necessity of the condition follows from \cite[Lemma 5.5]{Xu:Comb}. So it remains to prove its sufficiency. We will proceed by induction on $|J|$. If $|J| = 2$, this has been proved in \cite[Proposition 5.2]{Xu:Comb}. 

Suppose $|J| = m+1$. We first ``expand" $[B_{m+1}, A_{m+1}]$ to $[B^{*}_{m+1}, A^{*}_{m+1}]$ (cf. \cite[Section 7.2]{Xu:Comb}), so that $B_{m+1}^{*} = B_{m}$. By \cite[Proposition 7.4]{Xu:Comb}, we know $\r^{\Sigma_{0}}_{M, >_{\psi}}(\psi, \ul, \ueta) \neq 0$ if and only if
\begin{align}% expand Eq
\label{eq: expand}
\r^{\Sigma_{0}}_{M, >_{\psi}} \Big(\psi_{-}, \ul_{-}, \ueta_{-}; (\rho, A^{*}_{m+1}, B^{*}_{m+1}, l^{*}_{m+1}, \eta_{m+1}, \zeta_{m+1}) \Big) \neq 0
\end{align}
where $\psi_{-}$ is defined by
\[
Jord(\psi_{-}) = Jord(\psi) \backslash \{(\rho, A_{m+1}, B_{m+1}, \zeta_{m+1})\}
\]
and 
\[
l^{*}_{m+1} = l_{m+1} + (B_{m+1} - B_{m}).
\]
It is easy to check that the condition \eqref{eq: nonvanishing} holds for $\r^{\Sigma_{0}}_{M, >_{\psi}}(\psi, \ul, \ueta)$ if and only if it holds for the representation in \eqref{eq: expand}. So we will assume $B_{m+1} = B_{m}$ from now on.

Next we can ``pull" $[B_{m+1}, A_{m+1}], [B_{m}, A_{m}]$ (cf. \cite[7.1]{Xu:Comb}), so that they are far away from $\cup_{i < m}\{(\rho, A_{i}, B_{i}, \zeta_{i})\}$. By \cite[Proposition 7.1, 7.3]{Xu:Comb}, we know $\r^{\Sigma_{0}}_{M, >_{\psi}}(\psi, \ul, \ueta) \neq 0$ if the following representations are all nonzero. So it suffices to show each of them is nonzero by our induction assumption. Let $\psi_{-}$ be defined by
\[
Jord(\psi_{-}) = Jord(\psi) \backslash \{(\rho, A_{m+1}, B_{m+1}, \zeta_{m+1}), (\rho, A_{m}, B_{m}, \zeta_{m})\}.
\]

\begin{enumerate}

\item Show 
\begin{align}% pull 1 Eq
\label{eq: pull 1}
\r^{\Sigma_{0}}_{M, >_{\psi}} \Big(\psi_{-}, \ul_{-}, \ueta_{-}; (\rho, A_{m+1} + T, B_{m+1} + T, l_{m+1}, \eta_{m+1}, \zeta_{m+1}), (\rho, A_{m} +T, B_{m} + T, l_{m}, \eta_{m}, \zeta_{m})\Big) \neq 0
\end{align}
for some $T$. Let $J_{-} = Jord(\psi_{-})$. Then we will choose $T$ so that $J^{c}_{-} \gg J_{-}$. To make $J^{c}_{-}$ having discrete diagonal restriction, we will shift $[B_{m+1} + T, A_{m+1} + T]$ further to $[B_{m+1} + T', A_{m+1} + T']$ such that $B_{m+1} + T' > A_{m} +T$. Then by our induction assumption,
\[
\r^{\Sigma_{0}}_{M, >_{\psi}} \Big(\psi_{-}, \ul_{-}, \ueta_{-}; (\rho, A_{m+1} + T', B_{m+1} + T', l_{m+1}, \eta_{m+1}, \zeta_{m+1}), (\rho, A_{m} +T, B_{m} + T,  l_{m}, \eta_{m}, \zeta_{m}) \Big) \neq 0
\]
Let $\psi_{\gg}$ be the dominating parameter with discrete diagonal restriction, obtained by shifting $[B_{i}, A_{i}]$ to $[B_{i} + T_{i}, A_{i} + T_{i}]$ with $A_{i} + T_{i} < B_{m} + T$ for all $1 \leqslant i \leqslant m - 1$. Then
\begin{align*}
\r^{\Sigma_{0}}_{M, >_{\psi}}(\psi_{\gg}, \ul, \ueta) \hookrightarrow \times_{i < m} \begin{pmatrix}
              \zeta_{i}(B_{i} + T_{i}) & \cdots & \zeta_{i}(B_{i} + 1) \\
              \vdots &  & \vdots \\
              \zeta_{i}(A_{i} + T_{i}) & \cdots & \zeta_{i}(A_{i} + 1) \end{pmatrix} \rtimes \r^{\Sigma_{0}}_{M, >_{\psi}}\Big(\psi_{-}, \ul_{-}, \ueta_{-}; \\
(\rho, A_{m+1} + T', B_{m+1} + T', l_{m+1}, \eta_{m+1}, \zeta_{m+1}), (\rho, A_{m} +T, B_{m} + T,  l_{m}, \eta_{m}, \zeta_{m}) \Big)
\end{align*}
By \cite[Proposition 5.2]{Xu:Comb}, 
\[
\Jac_{(\rho, A_{m+1} + T', B_{m+1} + T', \zeta_{m+1}) \mapsto (\rho, A_{m+1} + T, B_{m+1} + T, \zeta_{m+1})} \r^{\Sigma_{0}}_{M, >_{\psi}}(\psi_{\gg}, \ul, \ueta) \neq 0.
\]
So after we apply the same Jacquet functor to the full induced representation above, we should get something nonzero. Since $B_{m+1} + T + 1 > A_{i} + T_{i}$ for $i < m$, the result is 
\begin{align*}
& \times_{i < m} \begin{pmatrix}
              \zeta_{i}(B_{i} + T_{i}) & \cdots & \zeta_{i}(B_{i} + 1) \\
              \vdots &  & \vdots \\
              \zeta_{i}(A_{i} + T_{i}) & \cdots & \zeta_{i}(A_{i} + 1) \end{pmatrix} \rtimes \Jac_{(\rho, A_{m+1} + T', B_{m+1} + T', \zeta_{m+1}) \mapsto (\rho, A_{m+1} + T, B_{m+1} + T, \zeta_{m+1})} \\ 
& \r^{\Sigma_{0}}_{M, >_{\psi}} \Big(\psi_{-}, \ul_{-}, \ueta_{-};
(\rho, A_{m+1} + T', B_{m+1} + T', l_{m+1}, \eta_{m+1}, \zeta_{m+1}), (\rho, A_{m} +T, B_{m} + T,  l_{m}, \eta_{m}, \zeta_{m}) \Big) \neq 0
\end{align*}
This shows \eqref{eq: pull 1}.

\item Show 
\begin{align}% pull 2 Eq
\label{eq: pull 2}
\r^{\Sigma_{0}}_{M, >_{\psi}} \Big(\psi_{-}, \ul_{-}, \ueta_{-}; (\rho, A_{m+1} + T, B_{m+1} + T, l_{m+1}, \eta_{m+1}, \zeta_{m+1}), (\rho, A_{m}, B_{m}, l_{m}, \eta_{m}, \zeta_{m})\Big) \neq 0
\end{align}
for some $T$. Let $J_{-} = Jord(\psi_{-}) \sqcup \{(\rho, A_{m}, B_{m}, \zeta_{m})\}$. We can choose $T$ so that $J_{-}^{c} \gg J_{-}$. Then the statement follows from our induction assumption immediately.

\item Show 
\begin{align}% pull 3 Eq
\label{eq: pull 3}
\r^{\Sigma_{0}}_{M, >'_{\psi}}\Big(\q_{-}, \ul'_{-}, \ueta'_{-}; (\rho, A_{m+1}, B_{m+1}, l'_{m+1}, \eta'_{m+1}, \zeta_{m+1}), (\rho, A_{m} + T, B_{m} +T , l'_{m}, \eta'_{m}, \zeta_{n-1})\Big) \neq 0
\end{align}
for some $T$, where $>'_{\psi}$ is obtained by switching $(\rho, A_{m+1}, B_{m+1}, \zeta_{m+1})$ with $(\rho, A_{m}, B_{m}, \zeta_{m})$, and $(\ul', \ueta') = S^{+}_{m+1}(\ul, \ueta)$ (cf. \cite[Section 6.1]{Xu:Comb}) given by the change of order formula. Let $J_{-} = Jord(\psi_{-}) \sqcup \{(\rho, A_{m+1}, B_{m+1}, \zeta_{m+1})\}$. We can choose $T$ so that $J_{-}^{c} \gg J_{-}$. Then the statement follows from our induction assumption again, provided we can verify the representation in \eqref{eq: pull 3} satisfies \eqref{eq: nonvanishing}. Indeed, we only need to show 
\begin{align}% transmit Eq
\label{eq: transmit}
\begin{cases}
\eta'_{m+1} = (-1)^{A_{m-1} - B_{m-1}} \eta_{m-1}    & \Rightarrow A_{m +1} - l'_{m +1} \geqslant A_{m-1} - l_{m-1}, \quad B_{m +1} + l'_{ m+1} \geqslant B_{m-1} + l_{m-1},  \\
\eta'_{m+1} \neq (-1)^{A_{m-1} - B_{m-1}} \eta_{m-1}  & \Rightarrow B_{m +1} + l'_{m +1} > A_{m-1} - l_{m-1}.   
\end{cases}  
\end{align}
We leave it to the next lemma.
\end{enumerate}
\end{proof}

\begin{lemma}
\eqref{eq: transmit} holds.
\end{lemma}

\begin{proof}
We divide into three cases according to the change of order formula.
\begin{enumerate}

\item If $\eta_{m+1} \neq (-1)^{A_{m} - B_{m}} \eta_{m}$, then 
\[
\begin{cases}
\eta'_{m + 1} = \eta_{m} \\
l'_{m+1} = (B_{m} + l_{m}) - (A_{m} - l_{m}) + l_{m+1} - 1
\end{cases}
\]
We get
\begin{align*}
B_{m+1} + l'_{m+1} = (B_{m+1} + l_{m+1}) + (B_{m} + l_{m}) - (A_{m} - l_{m}) - 1 \\
A_{m+1} - l'_{m+1} = (A_{m+1} - l_{m+1}) + (A_{m} - l_{m}) - (B_{m} + l_{m}) + 1
\end{align*}
By \eqref{eq: nonvanishing}, we have 
\[
B_{m+1} + l_{m+1} > A_{m} - l_{m}.
\]

\begin{enumerate}

\item When $\eta_{m} \neq (-1)^{A_{m-1} - B_{m-1}} \eta_{m-1}$, then $\eta'_{m + 1} \neq (-1)^{A_{m-1} - B_{m-1}} \eta_{m-1}$. We need to show
\[
B_{m+1} + l'_{m+1} > A_{m-1} - l_{m-1}.
\]
By \eqref{eq: nonvanishing}, we have 
\[
B_{m} + l_{m} > A_{m-1} - l_{m-1}.
\]
Then
\[
B_{m+1} + l'_{m+1} \geqslant B_{m} + l_{m} > A_{m-1} - l_{m-1}.
\]

\item When $\eta_{m} = (-1)^{A_{m-1} - B_{m-1}} \eta_{m-1}$, then $\eta'_{m + 1} = (-1)^{A_{m-1} - B_{m-1}} \eta_{m-1}$. We need to show
\[
\begin{cases}
B_{m+1} + l'_{m+1}  \geqslant B_{m-1} + l_{m-1} \\
A_{m+1} - l'_{m+1}  \geqslant A_{m-1} - l_{m-1}
\end{cases}
\]
By \eqref{eq: nonvanishing}, we have 
\[
\begin{cases}
B_{m} + l_{m}  \geqslant B_{m-1} + l_{m-1} \\
A_{m} - l_{m}  \geqslant A_{m-1} - l_{m-1}
\end{cases}
\]
Then 
\begin{align*}
B_{m+1} + l'_{m+1} & \geqslant B_{m} + l_{m} \geqslant B_{m-1} + l_{m-1} \\
A_{m+1} - l'_{m+1} & \geqslant A_{m+1} - l_{m+1} \geqslant A_{m} - l_{m} \geqslant A_{m-1} - l_{m-1}
\end{align*}

\end{enumerate}

\item If $\eta_{m+1} = (-1)^{A_{m} - B_{m}} \eta_{m}$ and
\[
l_{m+1} - l_{m} < (A_{m+1} - B_{m+1})/2 - (A_{m} - B_{m}) + l_{m},
\] 
then 
\[
\begin{cases}
\eta'_{m + 1} \neq \eta_{m} \\
l'_{m+1} = (A_{m} - l_{m}) - (B_{m} + l_{m})  + l_{m+1} - 1
\end{cases}
\]
We get
\begin{align*}
B_{m+1} + l'_{m+1} = (B_{m+1} + l_{m+1}) - (B_{m} + l_{m}) + (A_{m} - l_{m}) + 1 \\
A_{m+1} - l'_{m+1} = (A_{m+1} - l_{m+1}) - (A_{m} - l_{m}) + (B_{m} + l_{m}) - 1
\end{align*}
By \eqref{eq: nonvanishing}, we have 
\[
\begin{cases}
B_{m+1} + l_{m+1}  \geqslant B_{m} + l_{m} \\
A_{m+1} - l_{m+1}  \geqslant A_{m} - l_{m}
\end{cases}
\]

\begin{enumerate}

\item When $\eta_{m} \neq (-1)^{A_{m-1} - B_{m-1}} \eta_{m-1}$, then $\eta'_{m + 1} = (-1)^{A_{m-1} - B_{m-1}} \eta_{m-1}$. We need to show
\[
\begin{cases}
B_{m+1} + l'_{m+1}  \geqslant B_{m-1} + l_{m-1} \\
A_{m+1} - l'_{m+1}  \geqslant A_{m-1} - l_{m-1}
\end{cases}
\]
By \eqref{eq: nonvanishing}, we have 
\[
B_{m} + l_{m} > A_{m-1} - l_{m-1}.
\]
Then
\begin{align*}
B_{m+1} + l'_{m+1} & \geqslant (A_{m} - l_{m}) + 1 \geqslant (A_{m-1} - l_{m-1}) + 1 \geqslant B_{m-1} + l_{m-1} \\
A_{m+1} - l'_{m+1}  & \geqslant (B_{m} + l_{m}) - 1 \geqslant  A_{m-1} - l_{m-1}
\end{align*}

\item When $\eta_{m} = (-1)^{A_{m-1} - B_{m-1}} \eta_{m-1}$, then $\eta'_{m + 1} \neq (-1)^{A_{m-1} - B_{m-1}} \eta_{m-1}$. We need to show
\[
B_{m+1} + l'_{m+1} > A_{m-1} - l_{m-1} 
\]
By \eqref{eq: nonvanishing}, we have 
\[
\begin{cases}
B_{m} + l_{m}  \geqslant B_{m-1} + l_{m-1} \\
A_{m} - l_{m}  \geqslant A_{m-1} - l_{m-1}
\end{cases}
\]
Then 
\[
B_{m+1} + l'_{m+1} \geqslant (A_{m} - l_{m}) + 1 \geqslant (A_{m-1} - l_{m-1}) + 1 > A_{m-1} - l_{m-1}
\]

\end{enumerate}

\item If $\eta_{m+1} = (-1)^{A_{m} - B_{m}} \eta_{m}$ and
\[
l_{m+1} - l_{m} \geqslant (A_{m+1} - B_{m+1})/2 - (A_{m} - B_{m}) + l_{m},
\] 
then 
\[
\begin{cases}
\eta'_{m + 1} = \eta_{m} \\
l'_{m+1} = (A_{m+1} - B_{m+1}) - l_{m+1} - (A_{m} - l_{m})  + (B_{m} + l_{m})
\end{cases}
\]
We get
\begin{align*}
B_{m+1} + l'_{m+1} = (A_{m+1} - l_{m+1}) - (A_{m} - l_{m}) + (B_{m} + l_{m}) \\
A_{m+1} - l'_{m+1} = (B_{m+1} + l_{m+1}) - (B_{m} + l_{m}) + (A_{m} - l_{m})
\end{align*}
By \eqref{eq: nonvanishing}, we have 
\[
\begin{cases}
B_{m+1} + l_{m+1}  \geqslant B_{m} + l_{m} \\
A_{m+1} - l_{m+1}  \geqslant A_{m} - l_{m}
\end{cases}
\]

\begin{enumerate}

\item When $\eta_{m} \neq (-1)^{A_{m-1} - B_{m-1}} \eta_{m-1}$, then $\eta'_{m + 1} \neq (-1)^{A_{m-1} - B_{m-1}} \eta_{m-1}$. We need to show
\[
B_{m+1} + l'_{m+1} > A_{m-1} - l_{m-1} 
\]
By \eqref{eq: nonvanishing}, we have 
\[
B_{m} + l_{m} > A_{m-1} - l_{m-1}.
\]
Then
\[
B_{m+1} + l'_{m+1} \geqslant B_{m} + l_{m} > A_{m-1} - l_{m-1}
\]

\item When $\eta_{m} = (-1)^{A_{m-1} - B_{m-1}} \eta_{m-1}$, then $\eta'_{m + 1} = (-1)^{A_{m-1} - B_{m-1}} \eta_{m-1}$. We need to show
\[
\begin{cases}
B_{m+1} + l'_{m+1} \geqslant B_{m-1} + l_{m-1} \\
A_{m+1} - l'_{m+1}  \geqslant A_{m-1} - l_{m-1}
\end{cases}
\]
By \eqref{eq: nonvanishing}, we have 
\[
\begin{cases}
B_{m} + l_{m}  \geqslant B_{m-1} + l_{m-1} \\
A_{m} - l_{m}  \geqslant A_{m-1} - l_{m-1}
\end{cases}
\]
Then
\begin{align*}
B_{m+1} + l'_{m+1} \geqslant B_{m} + l_{m} \geqslant B_{m-1} + l_{m-1} \\
A_{m+1} - l'_{m+1}  \geqslant A_{m} - l_{m} \geqslant A_{m-1} - l_{m-1}
\end{align*}

\end{enumerate}

\end{enumerate}

\end{proof}

More generally, we can drop the assumption \eqref{eq: fix rho}, but only assume $\q = \q_{p}$. Suppose for each $\rho$ appearing in $Jord(\q)$, we have the same setup as in Theorem~\ref{thm: nonvanishing}. Then we have

\begin{theorem}% nonvanishing 2
\label{thm: nonvanishing 2}
$\r^{\Sigma_{0}}_{M, >_{\psi}}(\psi, \ul, \ueta) \neq 0$ if and only if the condition \eqref{eq: nonvanishing} is satisfied for each $\rho$.
\end{theorem}

\begin{proof}
We can apply the arguments of the proof of Theorem~\ref{thm: nonvanishing} to each $\rho$ one by one, which reduces it to the case that $|J| = 2$ for each $\rho$. Then this case follows from \cite[Proposition 5.3]{Xu:Comb}.
\end{proof}

\section{Change sign}

In this appendix, we would like to extend \cite[Proposition 7.6]{Xu:Comb} as follows. Let $\q$ be an Arthur parameter of $G(F)$ such that $\q = \q_{p}$. We choose an admissible order $>_{\q}$ and fix an irreducible unitary supercuspidal representation $\rho$ of $GL(d_{\rho}, F)$. Let us index the Jordan blocks in $Jord_{\rho}(\q)$ such that 
\[
(\rho, A_{i+1}, B_{i+1}, \zeta_{i+1}) >_{\q} (\rho, A_{i}, B_{i}, \zeta_{i}).
\]
Suppose there exists $n$ such that for $i > n$, 
\[
(\rho, A_{i}, B_{i}, \zeta_{i}) \gg \cup_{j=1}^{n}\{(\rho, A_{j}, B_{j}, \zeta_{j})\}.
\]
Moreover, there exists $1 \leqslant m \leqslant n$ such that 
\[
A_m = \cdots = A_1 \geqslant A_i, \quad  B_m = \cdots = B_1 = 1/2, \quad \zeta_{m} = \cdots = \zeta_{1} \neq \zeta_{i}
\] 
for $m < i \leqslant n$. We introduce another parameter $\psi^{*}$ by changing $(\rho, A_{i}, B_{i}, \zeta_{i})$ to $(\rho, A_{i} + 1, B_{i}, -\zeta_{i})$ for $i \leqslant m$.
%introduce a new order $>_{\psi^{*}}$ by reversing the order between $(i+1)$-th and $i$-th Jordan blocks for $i < m$, i.e.,
%\[
%(\rho, A_{i} + 1, B_{i}, -\zeta_{i}) >_{\q^*} (\rho, A_{i+1} + 1, B_{i+1}, -\zeta_{i+1}).
%\]
For any $(\ul, \ueta)$, such that 
\begin{align}% change sign 1 Eq
\label{eq: change sign 1}
\quad l_{i+1} = l_{i}, \quad \eta_{i+1} = (-1)^{A_{i} - \frac{1}{2}}\eta_{i} \quad \text{ for } i < m,
\end{align}
we can associate it with $(\ul^*, \ueta^*)$, defined as follows. For $i > m$,
\[
l^{*}_{i} = l_{i}, \quad \eta^{*}_{i} = \eta_{i}.
\]
For $i < m$,
\begin{align}% change sign 2 Eq
\label{eq: change sign 2}
l^{*}_{i+1} = l^{*}_{i}, \quad \eta^{*}_{i+1} = (-1)^{A_{i} + \frac{1}{2}} \eta^{*}_{i}.
\end{align}
Then it remains to specify $l^{*}_1, \eta^{*}_1$, which are given by
\[
\eta^{*}_{1} = -\eta_{1}, \quad l^{*}_{1} = \begin{cases} 
l_{1} + 1 & \text{ if } \eta_{1} = 1 \\
l_{1} & \text{ if } \eta_{1} = -1 \\
\end{cases}
\]
In case $l_{1} = (A_1 + \frac{1}{2})/2$, we fix $\eta_{1} = -1$.

\begin{proposition}% change sign PROPOSITION
\label{prop: change sign}
$\pi^{\Sigma_{0}}_{M, >_{\psi}}(\psi, \ul, \ueta) \neq 0$ if and only if $\pi^{\Sigma_{0}}_{M, >_{\psi}}(\psi^{*}, \ul^{*}, \ueta^{*}) \neq 0$. Moreover, 
\begin{align}% change sign Eq
\label{eq: change sign}
\r^{\Sigma_{0}}_{M, >_{\q}}(\q^{*}, \ul^{*}, \ueta^{*}) \hookrightarrow \times_{i = 1}^{m} \langle -\zeta_{i}1/2, \cdots, -\zeta_{i}(A_{i} + 1) \rangle \rtimes  \r^{\Sigma_{0}}_{M, >_{\q}}(\q, \ul, \ueta)      
\end{align}
\end{proposition}

\begin{remark}
\cite[Proposition 7.6]{Xu:Comb} settles the case when $m = 1$.
\end{remark}

\begin{proof}
As in the proof of \cite[Proposition 7.6]{Xu:Comb}, we can reduce it to the case that 
\[
m = n \text{ and $Jord(\q) \backslash \cup_{i=1}^{n}\{(\rho, A_{i}, B_{i}, \zeta_{i})\}$ has discrete diagonal restriction}.
\]
Because of the conditions \eqref{eq: change sign 1} and \eqref{eq: change sign 2}, we have
\[
\pi^{\Sigma_{0}}_{M, >_{\psi}}(\psi, \ul, \ueta) \neq 0 \text{ and } \pi^{\Sigma_{0}}_{M, >_{\psi}}(\psi^{*}, \ul^{*}, \ueta^{*}) \neq 0
\] 
by Theorem~\ref{thm: nonvanishing 2}. So we only need to show \eqref{eq: change sign}, and we will proceed by induction on $n$. 

Let $\psi^{*}_{>}$ be obtained from $\psi^{*}$ by changing $(\rho, A_{n} + 1, 1/2, -\zeta_{n})$ to $(\rho, A_{n} + 1 + T_{n}, 1/2 + T_{n}, -\zeta_{n})$ for $T_{n}$ sufficiently large. Then by our induction assumption, we have
\[
\pi^{\Sigma_{0}}_{M, >_{\psi}}(\psi^{*}_{>}, \ul^{*}, \ueta^{*}) \hookrightarrow \times_{i = 1}^{n-1} \langle -\zeta_{i}1/2, \cdots, -\zeta_{i}(A_{i} + 1) \rangle \rtimes \pi^{\Sigma_{0}}_{M, >_{\psi}}(\psi^{(n)}_{>}, \ul^{(n)}, \ueta^{(n)})
\]
where $\psi^{(n)}_{>}$ is obtained from $\psi^{*}_{>}$ by changing $(\rho, A_{i} + 1, 1/2, -\zeta_{i})$ back to $(\rho, A_{i}, 1/2, \zeta_{i})$ for $1 \leqslant i < n$. Moreover, 
\[
l^{(n)}_{i} = l_{i}, \quad \eta^{(n)}_{i} = \eta_{i} \quad \text{ for } i < n,
\]
and 
\[
l^{(n)}_{i} = l^{*}_{i}, \quad \eta^{(n)}_{i} = \eta^{*}_{i} \quad \text{ for } i \geqslant n.
\]
Then we claim 
\begin{align}% change sign 4 Eq
\label{eq: change sign 4}
\pi^{\Sigma_{0}}_{M, >_{\psi}}(\psi^{(n)}_{>}, \ul^{(n)}, \ueta^{(n)}) \hookrightarrow \underbrace{\begin{pmatrix}
              -\zeta_{n}(1/2 + T_{n}) & \cdots & -\zeta_{n}1/2 \\
              \vdots &  & \vdots \\
              -\zeta_{n}(A_{n} + 1 + T_{n}) & \cdots & -\zeta_{n}(A_{n} + 1)
       \end{pmatrix}}_{\mathcal{C}_{X_{n}}} \rtimes \pi^{\Sigma_{0}}_{M, >_{\psi}}(\psi, \ul, \ueta).
\end{align}
where
\[
X_{n} := \begin{bmatrix}
              -\zeta_{n}(1/2 + T_{n}) & \cdots & -\zeta_{n}1/2 \\
              \vdots &  & \vdots \\
              -\zeta_{n}(A_{n} + 1 + T_{n}) & \cdots & -\zeta_{n}(A_{n} + 1)
       \end{bmatrix}
\]
If this is the case, then
\begin{align*}
\pi^{\Sigma_{0}}_{>_{\psi}}(\psi^{*}_{>}, \ul^{*}, \ueta^{*})  & \hookrightarrow  \times_{i = l}^{n-1} \langle -\zeta_{i}1/2, \cdots, -\zeta_{i}(A_{i} + 1) \rangle \times \mathcal{C}_{X_{n}} \rtimes \pi^{\Sigma_{0}}_{M, >_{\psi}}(\psi, \ul, \ueta) \\
& \cong \mathcal{C}_{X_{n}} \times \times_{i = l}^{n-1} \langle -\zeta_{i}1/2, \cdots, -\zeta_{i}(A_{i} + 1) \rangle \rtimes \pi^{\Sigma_{0}}_{M, >_{\psi}}(\psi, \ul, \ueta),
\end{align*}
from which \eqref{eq: change sign} follows.

We still need to show the claim \eqref{eq: change sign 4}. Let $\psi^{(n)}$ be obtained from $\psi^{(n)}_{>}$ by moving $(\rho, A_{n} + 1 + T_{n}, 1/2 + T_{n}, -\zeta_{n})$ back to $(\rho, A_{n} + 1, 1/2, -\zeta_{n})$. Suppose $\pi^{\Sigma_{0}}_{M, >_{\psi}}(\psi^{(n)}, \ul^{(n)}, \ueta^{(n)}) \neq 0$, then 
\begin{align}% change sign 3 Eq
\label{eq: change sign 3}
\pi^{\Sigma_{0}}_{M, >_{\psi}}(\psi^{(n)}_{>}, \ul^{(n)}, \ueta^{(n)}) \hookrightarrow \begin{pmatrix}
              -\zeta_{n}(1/2 + T_{n}) & \cdots & -\zeta_{n}3/2 \\
              \vdots &  & \vdots \\
              -\zeta_{n}(A_{n} + 1 + T_{n}) & \cdots & -\zeta_{n}(A_{n} + 2)
       \end{pmatrix} \rtimes \pi^{\Sigma_{0}}_{M, >_{\psi}}(\psi^{(n)}, \ul^{(n)}, \ueta^{(n)}).
\end{align}
To show the nonvanishing of $\pi^{\Sigma_{0}}_{M, >_{\psi}}(\psi^{(n)}, \ul^{(n)}, \ueta^{(n)})$, we need to switch to a new order $>'_{\psi}$ by moving $(\rho, A_{n} + 1, 1/2, -\zeta_{n})$ to the last position. Then
\[
\pi^{\Sigma_{0}}_{M, >_{\psi}}(\psi^{(n)}, \ul^{(n)}, \ueta^{(n)}) = \pi^{\Sigma_{0}}_{M, >'_{\psi}}(\psi^{(n)}, \ul^{'(n)}, \ueta^{'(n)}),
\]
where 
\[
l^{'(n)}_{i} = l_{i}^{(n)}, \quad \eta^{'(n)}_{i} = \eta^{(n)}_{i} \quad \text{ for } i > n,
\]
and
\[
l^{'(n)}_{i} = l_{i}^{(n)}, \quad \eta^{'(n)}_{i} = (-1)^{A_{n} - 1/2}\eta^{(n)}_{i} \quad \text{ for } i < n,
\]
and
\[
l^{'(n)}_{n} = l^{*}_{1}, \quad \eta^{'(n)}_{n} = \eta^{*}_{1}. 
\]
Let $\psi^{(n)}_{\gg}$ be a dominating parameter  for $\psi^{(n)}$ with respect to $>'_{\psi}$, obtained by changing $(\rho, A_{i}, B_{i}, \zeta_{i})$ to $(\rho, A_{i} + T_{i}, B_{i} + T_{i}, \zeta_{i})$ for $i < n$. We also require that $\psi^{(n)}_{\gg}$ has discrete diagonal restriction. Then by \cite[Proposition 7.6]{Xu:Comb},
\[
\pi^{\Sigma_{0}}_{M, >'_{\psi}}(\psi^{(n)}_{\gg}, \ul^{'(n)}, \ueta^{'(n)}) \hookrightarrow \langle -\zeta_{n}1/2, \cdots, -\zeta_{n}(A_{n} + 1) \rangle \rtimes \pi^{\Sigma_{0}}_{M, >'_{\psi}}(\psi_{\gg}, \ul', \ueta'),
\]
where $\psi_{\gg}$ is obtained from $\psi^{(n)}_{\gg}$ by changing $(\rho, A_{n} + 1, 1/2, -\zeta_{n})$ back to $(\rho, A_{n}, 1/2, \zeta_{n})$. Note 
\[
l'_{i} = l^{'(n)}_{i}, \quad \eta'_{i} = \eta^{'(n)}_{i}  \quad \text{ for } i \neq n,
\]
and
\[
l'_{n} = l_{1}, \quad \eta'_{n} = \eta_{1}.
\]
It is easy to check by the change of order formula that
\[
\pi^{\Sigma_{0}}_{M, >_{\psi}}(\psi, \ul, \ueta) = \pi^{\Sigma_{0}}_{M, >'_{\psi}}(\psi, \ul', \ueta').
\]
In particular, the right hand side is nonzero. Therefore,
\[
\pi^{\Sigma_{0}}_{M, >'_{\psi}}(\psi_{\gg}, \ul', \ueta') \hookrightarrow \times_{i = 1}^{n-1} \begin{pmatrix}
              \zeta_{i}(1/2 + T_{i}) & \cdots & \zeta_{i}3/2 \\
              \vdots &  & \vdots \\
              \zeta_{i}(A_{i} + T_{i}) & \cdots & \zeta_{i}(A_{i} + 1)
       \end{pmatrix} \rtimes \pi^{\Sigma_{0}}_{M, >'_{\psi}}(\psi, \ul', \ueta')
\]
Combined with the previous inclusion, we get
\begin{align*}
\pi^{\Sigma_{0}}_{M, >'_{\psi}}(\psi^{(n)}_{\gg}, \ul^{'(n)}, \ueta^{'(n)}) & \hookrightarrow \langle -\zeta_{n}1/2, \cdots, -\zeta_{n}(A_{n} + 1) \rangle  \times \\
&  \times_{i = 1}^{n-1} \begin{pmatrix}
              \zeta_{i}(1/2 + T_{i}) & \cdots & \zeta_{i}3/2 \\
              \vdots &  & \vdots \\
              \zeta_{i}(A_{i} + T_{i}) & \cdots & \zeta_{i}(A_{i} + 1)
       \end{pmatrix} \rtimes \pi^{\Sigma_{0}}_{M, >'_{\psi}}(\psi, \ul', \ueta') \\
& \cong \times_{i = 1}^{n-1} \begin{pmatrix}
              \zeta_{i}(1/2 + T_{i}) & \cdots & \zeta_{i}3/2 \\
              \vdots &  & \vdots \\
              \zeta_{i}(A_{i} + T_{i}) & \cdots & \zeta_{i}(A_{i} + 1)
       \end{pmatrix}  \\
& \times \langle -\zeta_{n}1/2, \cdots, -\zeta_{n}(A_{n} + 1) \rangle \rtimes \pi^{\Sigma_{0}}_{M, >'_{\psi}}(\psi, \ul', \ueta') 
\end{align*}
Consequently, $\pi^{\Sigma_{0}}_{M, >'_{\psi}}(\psi^{(n)}, \ul^{'(n)}, \ueta^{'(n)}) \neq 0$ and 
\[
\pi^{\Sigma_{0}}_{M, >'_{\psi}}(\psi^{(n)}, \ul^{'(n)}, \ueta^{'(n)}) \hookrightarrow \langle -\zeta_{n}1/2, \cdots, -\zeta_{n}(A_{n} + 1) \rangle  \rtimes \pi^{\Sigma_{0}}_{M, >'_{\psi}}(\psi, \ul', \ueta').
\]
Substitute the above expression into \eqref{eq: change sign 3}, we obtain
\begin{align*}
\pi^{\Sigma_{0}}_{M, >_{\psi}}(\psi^{(n)}_{>}, \ul^{(n)}, \ueta^{(n)}) & \hookrightarrow \begin{pmatrix}
              -\zeta_{n}(1/2 + T_{n}) & \cdots & -\zeta_{n}3/2 \\
              \vdots &  & \vdots \\
              -\zeta_{n}(A_{n} + 1 + T_{n}) & \cdots & -\zeta_{n}(A_{n} + 2)
       \end{pmatrix} \\
& \times \langle -\zeta_{n}1/2, \cdots, -\zeta_{n}(A_{n} + 1) \rangle  \rtimes \pi^{\Sigma_{0}}_{M, >_{\psi}}(\psi, \ul, \ueta).
\end{align*}
Note the Jordan blocks in $Jord_{\rho}(\psi)$ satisfies $A_{i} < A_{n} + 1$ for $i \leqslant n$, and $B_{i} > A_{n} + 1 + T_{n}$ for $i > n$. If we apply $\Jac_{X_{n}}$ to the right hand side of the above expression, we can only get $\pi^{\Sigma_{0}}_{M, >_{\psi}}(\psi, \ul, \ueta)$. This means the left hand side is the unique irreducible subrepresentation of the right hand side. Therefore, 
\[
\pi^{\Sigma_{0}}_{M, >_{\psi}}(\psi^{(n)}_{>}, \ul^{(n)}, \ueta^{(n)}) \hookrightarrow \mathcal{C}_{X_{n}} \rtimes \pi^{\Sigma_{0}}_{M, >_{\psi}}(\psi, \ul, \ueta),
\]
which is exactly \eqref{eq: change sign 4}. This finishes our proof.

\end{proof}

\bibliographystyle{amsalpha}

\bibliography{reps}

\end{document}